\documentclass[12pt,notitlepage]{amsart}
\usepackage{latexsym,amsfonts,amssymb,amsmath,amsthm} 
\newcommand{\N}{\ensuremath{\mathbb N}}
\newcommand{\Z}{\ensuremath{\mathbb Z}}
\newcommand{\R}{\ensuremath{\mathbb R}}

\newcommand{\xbar}{\overline{x}}
\newcommand{\ybar}{\overline{y}}
\newcommand{\dd}{d^{\circ}}

\makeatletter
\newcommand{\vast}{\bBigg@{4}}
\newcommand{\Vast}{\bBigg@{5}}
\makeatother

\theoremstyle{plain}		
	\newtheorem{theorem}{Theorem}[section]
	\newtheorem{prop}[theorem]{Proposition}
	\newtheorem{cor}[theorem]{Corollary}
     \newtheorem{lemma}[theorem]{Lemma}
	\newtheorem{definition}[theorem]{Definition}

\theoremstyle{remark}		
	
	\newtheorem*{remark}{Remark}

\numberwithin{equation}{section}
\begin{document}
\title{Subconvexity for the Rankin-Selberg L-function in both levels}
\author{P. Edward Herman} 
\address{University of Chicago,Dept. of Mathematics,
5734 S. University Avenue,
Chicago, Illinois 60637}
\email{peherman@math.uchicago.edu}

\begin{abstract}
  In this paper, we obtain a subconvexity result for the Rankin-Selberg L-function in both levels. The new feature in this result is applying an amplification method of Duke-Friedlander-Iwaniec to a double Petersson-Kuznetsov trace formula. 
  As the trace formula ranges over all the $GL_2$ spectrum, this subconvexity result is unconditional of which Hecke eigenforms are chosen in the Rankin-Selberg L-function.    \end{abstract}

\maketitle
\section{Introduction}

We apply techniques gained from a beyond endoscopy approach to the Rankin-Selberg L-function in \cite{H} to understanding subconvexity. In \cite{H}, we studied an average of Rankin-Selberg L-functions $L(f \times h,s).$ More specifically, for a  smooth compactly supported test function $g,$ we investigated $$ \sum_{n}  g(\frac{n}{X}) \sum_h \sum_f a_{n}(f)a_n (h)$$ as $X$ gets large, with $a_{n}(f),a_{n}(g)$ being weighted Fourier coefficients associated to holomorphic or Maass Hecke eigenforms $f,h.$ On applying Mellin inversion to this formula, we are studying averages of Rankin-Selberg L-functions. Understanding this sum is reminiscent to understanding an approximate functional equation for the Rankin-Selberg L-function. However the important difference, which is explained in more detail below, is that the parameter $X$ above has no apparent connection to the essential data (level, weight,..) of the L-function, while in the approximate function equation the parameter(s) which one investigates in subconvexity problems are intrinsically associated to $X.$ For example, the sum $$ \sum_{n}\frac{1}{\sqrt{n}}  g(\frac{n}{pq}) \sum_h \sum_f a_{n}(f)a_n (h)$$ over forms $f,h$ of level $p,q,$ respectively, is
roughly the approximate functional equation for the Rankin-Selberg L-function. In an approach to breaking convexity in both levels, one would naively consider this sum as $p,q \to \infty.$

Given this connection of beyond endoscopy and the approximate functional equation, we ask if we can learn anything more about subconvexity for the Rankin-Selberg L-function. The clear connection of the two above sums hints we should study subconvexity as the two levels of the associated automorphic forms of the L-function go to infinity. 

This question of subconvexity for the Rankin-Selberg L-function in the level aspect has been studied in \cite{DFI}, \cite{KMV},and \cite{MV}. However, in these cases, one of the levels remains fixed as the other goes to infinity. Recently in \cite{HM} the first case where both levels vary was announced. We state their result. Let $M$ be a positive square-free integer and $P$ a prime with $(M,P)=1,$ $P \sim M^{\eta}$ and $0<\eta< \frac{2}{21}.$ Then for two holomorphic newforms $f,h$ we have the bound $$L(f \times h, 1/2) \ll (MP)^{1/2+\epsilon}\left( \frac{1}{(MP)^{\frac{\eta}{2(1+\eta)}}}+\frac{1}{(MP)^{\frac{2-21\eta}{64(1+\eta)}}} \right).$$   

We make no such restrictions in this paper and associate the subconvexity of the Rankin-Selberg L-function in both levels to subconvexity bounds of Dirichlet L-functions. Our result is:

\begin{theorem}\label{theo}
Let $f,h$ be holomorphic or Maass Hecke eigenforms associated to $GL_2$ automorphic representations of prime level $p,q,$ with $p\neq q$ with 
nontrivial central characters $\chi,\psi,$ respectively. Then for any $\epsilon>0$ and a fixed but arbitrarily small $\delta>0$ as $p,q \to \infty,$  we have $$L(f \times h, 1/2) \ll (pq)^{\epsilon}\bigg((pq)^{\frac{3\theta}{2}-1/8}+(pq)^{3/8+\frac{\theta}{4}}+p^{3/8+\frac{\theta}{4}}q^{\frac{\theta}{4}-1/8}+q^{3/8+\frac{\theta}{4}}p^{\frac{\theta}{4}-1/8} +  (1+q^{\delta})(pq)^{\frac{3}{8}+\frac{\theta}{4}}\bigg).$$ 
  Here $\theta>0$ comes from a (sub)convexity bound for the Dirichlet L-functions $L(\chi,1/2)\ll p^{ \theta+\epsilon}$ and  $L(\psi,1/2)\ll q^{ \theta+\epsilon}.$ 
\end{theorem}

\begin{remark}

 The convexity bound for the Rankin-Selberg L-function in the level aspect is $$L(f \times h, 1/2) \ll (pq)^{\frac{1}{2}+\epsilon}.$$
 
\end{remark}

\begin{cor}
If $\chi$ and $\psi$ are real quadratic characters, then we can apply Theorem \ref{theo} with $\theta=\frac{1}{6}$ using the result of \cite{CI} to get the bound $$L(f \times h, 1/2) \ll (pq)^{\frac{10}{24}+\epsilon}.$$ If one of the characters is not real then we can apply the bound $\theta=\frac{3}{16}$ of \cite{B} to get $$L(f \times h, 1/2) \ll (pq)^{\frac{27}{64}+\epsilon}.$$
\end{cor}


\section{Details of Paper}

Inspired by \cite{DFI} and \cite{KMV}, we apply the amplification method. Ignoring convergence issues for now, we study \begin{equation}\label{eq:beg00}\sum_{f\in H(p,\chi)}|\sum_{l \leq L} x_l \lambda_f(l)|^2 \sum_{h \in H(q,\psi) } |L(f \times h,1/2)|^2,\end{equation} where $x_l$ are arbitrary complex coefficients which we optimize to ``amplify" a single Rankin-Selberg L-function. Using an approximate functional equation and then a double Kuznetsov trace formula similar to \cite{H}, we bound \eqref{eq:beg00} in terms of a function $F$ of $p,q,$ and $L.$ In order to bound just a single L-function, we need positivity to get to an estimate $$|\sum_{l \leq L} x_l \lambda_f(l)|^2 \sum_{h \in H(q,\psi) } |L(f \times h,1/2)|^2 \ll F(p,q,x_l,L).$$ Then choosing the $x_l$ to make $|\sum_{l \leq L} x_l \lambda_f(l)|^2$ as large as possible ensures that $$L(f \times h,1/2) \ll \left(\frac{F(p,q,x_l,L)}{|\sum_{l \leq L} x_l \lambda_f(l)|}\right)^{1/2}$$ is as small as possible. 

To the details of actually getting the bound $F(p,q,x_l,L)$ after applying the double Kuznetsov formula, we open up the Kloosterman sums and reorganize the $c_1,c_2$-sums and associated exponential sums to get terms more amenable to estimation similar to \cite{H}. 

One then has four sums needing estimation. Two come from the geometric sides of the trace formulas (labelled $c_1,c_2$), and the two sums from opening up $|L(f \times h,s)|^2$ with the approximate functional equation. Performing Poisson summation and Mellin inversion at various steps reduces the estimation to $GL_1$ Dirichlet L-functions. 


Another feature in the calculation is that one sees a shadow of the Selberg trace formula. 
We motivate this connection with a short heuristic. First, we denote by $H(p,\chi)$ the automorphic forms with level $p$ and central character $\chi.$ Ignoring test functions and convergence issues, the second moment calculation essential to this 
paper is: \begin{multline}\label{eq:appp}\sum_{f\in H(p,1),h \in H(q,1) }|L(1/2,f \times h)|^2\approx \\ \sum_{f\in H
(p,1),h \in H(q,1) } \sum_{m,n \geq 1} \frac{1}{(mn)^{1/2}} g(\frac{n}{pq})g'(\frac{m}{pq}) a_f(n)
a_f(m)a_h(n)a_h(m)\approx \\ \left(\frac{6\sqrt{pq}}{\pi^2} \int_{-\infty}^{\infty} \frac{g(t)}{t^{1/2}}dt  
\right) \sum_{m\geq 1} \frac{1}{m^{1/2}}g'(\frac{m}{pq})\left(\sum_{F \in H((p,q),1)} |a_F(m)|^2\right).\end{multline}

The second approximation is just using classical Rankin-Selberg theory on the $n$-sum, or one can also apply an analogous argument of \cite{H}. Another application of Rankin-Selberg theory to the $m$-sum gives as a main term for \eqref{eq:appp}:  \begin{equation}\label{eq:fina} \frac{6pq}{\pi^2} \widetilde{g}(1/2)\widetilde{g'}(1/2) \left(\sum_{F \in H((q,D),1)} Res_{s=1}L(1,F \times \overline{F}) \right)\end{equation}  with $\widetilde{g}$ denoting the Mellin transform of $g.$ But this term $Res_{s=1}L(1,F \times \overline{F})$ cancels the orthonormalization built into the Kuznetsov trace formula. This final formula is an unweighted trace formula, or a Selberg trace formula. Though in this paper we take nontrivial central characters $\chi \mod(p),\psi \mod(q),$ something similar to this phenomenon is still seen.  Fortunately recognizing the geometric side of the Selberg trace formula, in this fashion, has been encountered and studied by Rudnick in his thesis \cite{R}. There also seems to be some connection of interactions of the geometric sides of the Kuznetsov and Selberg trace formula in the analysis in \cite{SY}, which counts the number of closed geodesics up to a large parameter $X.$

Next, though we explicitly avoided Bessel functions in our initial Kuznetsov trace formulas--in estimating the archimedean integrals that occur in our estimation--we do see Bessel functions. However, they display themselves ``late" in the calculation, and we use standard asymptotics to deal with them. We say ``late" in the sense that one could approach the above kind of subconvexity problem in Theorem \ref{theo} using one Kuznetsov trace formula (for the $f$-sum) and one Voronoi summation for a specific automorphic form $h,$ very similar to \cite{KMV}. In the early stages of such a calculation (specifically after Voronoi summation), one encounters difficult arithmetic and Bessel function analysis. Another difficulty is that one would like to understand subconvexity in both levels for all $GL_2$ forms, so one would need a different Voronoi summation for each type of automorphic form $h$ (holomorphic, Maass, and Eisenstein). Doing a double Kuznetsov trace formula simplifies this analysis. In another paper \cite{H1}, we prove that the Kuznetsov (Petersson) trace formula implies Voronoi summation, so no information is lost in the approach of the current paper.  

Finally, in order to ensure positivity of \begin{equation}\label{eq:beg88}\sum_{f\in H(p,\chi)}h(V,t_f)|\sum_{l \leq L} x_l \lambda_f(l)|^2 \sum_{h \in H(q,\psi) }h(W,t_h) |L(1/2,f \times h)|^2,\end{equation}--where now we include the test functions that ensure convergence (as compared to \eqref{eq:beg00})--in section \ref{isomo} we discuss how we can choose $V,W$ such that the test functions $h(W,t_h),h(V,t_f)$ are ``large" for automorphic forms with eigenvalues (or weights) $\{t_{f^{\circ}}, t_{h^{\circ}}\}$ and arbitrarily small at any other spectral values.

\subsection{The Kloosterman-Kloosterman term}
When we apply the trace formula to both spectral sums, we encounter four different terms: the delta-delta term, the delta-Kloosterman term, Kloosterman-delta term, and the Kloosterman-Kloosterman term. The first three are relatively straightforward to deal with. However, the Kloosterman-Kloosterman term is much more difficult, and we explain the steps in making an estimate on it.

This term comes from after applying the Kuznetsov trace formula twice and is seen in the total geometric side of the trace formula in \eqref{eq:poisson}, \begin{multline}\label{eq:mmain}\sqrt{pq} \sum_{m\geq 1} \frac{1}{m^{1/2}}F'_{M'}(\frac{\dd m}{pq}) \sum_{\substack{c_1\equiv 0(p)\\c_2\equiv 0(q)}} \frac{1}{c_1 c_2}\sum_{x  (c_1)^{*}}
 \sum_{y  (c_2)^{*}} \chi(\overline{x})\psi(\overline{y})e(\frac{\overline{x}\dd m}{c_1}+\frac{\overline{y}\frac{l_1 l_2m}{\dd k^2}}{c_2}) \times \\ \bigg\{
 \sum_k
 \int_{-\infty}^{\infty} e(pqt(\frac{c_2 x + c_1 y -kc_1c_2}{c_1 c_2}))
 \frac{F_{M}(t)}{t^{1/2}}
 V(\frac{4\pi
\sqrt{tpq\dd m}}{c_1})  W(\frac{4\pi \sqrt{tpq\frac{l_1 l_2m}{\dd k^2}}}{c_2})dt.\end{multline}
The functions $F_M,F'_{M'}$ come from a partition of unity of the approximate functional equation of the Rankin-Selberg function. The functions $W,V$ come from our choice of test functions in the trace formula. Trivially \eqref{eq:mmain} has a bound of $(pq)^2,$ we will sketch how to get the bound  $\frac{p^{\theta_1+\epsilon}q^{\theta_2+\epsilon}}{\sqrt{pq}},$ where $\theta_1,\theta_2 \leq 1/4.$

Fix the $m$-sum for now and look at the $c_1,c_2,$ and $k$-sums. We will now rearrange terms following \cite{H}. Naively, the idea is for each integer $n$ finding all $c_1,c_2,x,y,k$ such that $n=c_2 x + c_1 y -kc_1c_2.$ Rearranging the sums in this way, we can write $\overline{x} (c_1)$ and $\overline{y}(c_2)$ in terms of a paramter $r(n).$ Specifically, for $x(c_1)$ we have $\overline{x}=\frac{c_1 r+c_2}{n}(c_1).$ The advantage of this is that the difficult Kloosterman sums are essentially removed. Now again fixing $m,$ we are looking at the sum \begin{multline}\label{eq:kk08}  \sqrt{pq} \sum_{n \in \Z} \sum_{r(n)} e(\frac{d^{\circ}mr+\frac{l_1 l_2m}{d^{\circ}k^2}\overline{r}}{n})\overline{\chi(n)\psi(n)}\sum_{d|n}  \sum_{\substack{c_1\equiv 0(p)\\ (c_1,n/d)=1}} \sum_{ \substack{ c_2 = -(c_1\,r + \lambda\frac{n}d), c_2\equiv 0(q)\\(\lambda,c_1d)=1 }} \chi(dc_2)\psi(dc_1)H_n(\frac{dc_1}{pq},
\frac{dc_2}{pq}),\end{multline} which is similar to our \eqref{eq:kk}. Here $H_n$ is nice test function defined in terms of the test functions $V,W$ in \eqref{eq:defhn}. Denote the Mellin transform of $H_n$ by $\widetilde{H_n}.$ Then we can write \eqref{eq:kk08} using Mobius inversion as \begin{multline}\label{eq:blub}\frac{\sqrt{pq}}{(pq)^2} \sum_{n \in \Z} \sum_{r(n)} e(\frac{d^{\circ}mr+\frac{l_1 l_2m}{{\dd}^{\circ}k^2}\overline{r}}{n})\overline{\chi(n)\psi(n)}\sum_{d|n}  \sum_{\substack{c_1\equiv 0(p)\\ (c_1,n/d)=1}}\psi(dc_1) \sum_{f|c_1} \mu(f)\sum_{( \lambda,c_1d)=1 }
 \chi(-f\lambda qn)\times  \\ \int_{({s_1})=10} \int_{({s_2})=10} \widetilde{H_n}(s_1,s_2) (\frac{pq}{dc_1})^{s_1}   (\frac{pq}{dqf(-c_1r+\frac{\lambda n}{d})})^{s_2}ds_1ds_2.\end{multline} The first step in our estimation is that we can truncate the $n$-sum with a negligible error (Lemma \ref{ncc}) and get the following proposition 

\begin{prop}[Step one of estimate]\label{bigskew}
Let $p,q$ be distinct primes and $\delta,\theta>0$ be small but fixed. Let $H_n(x)$ be defined as in \eqref{eq:defhn}. We have \begin{multline} \label{eq:pp188}\frac{\sqrt{pq}}{(pq)^2} \sum_{\substack{n\neq 0\\ n \ll (pq)^\theta L^2}} \sum_{r(n)} e(\frac{d^{\circ}mr+\frac{l_1 l_2m}{{\dd}^{\circ}k^2}\overline{r}}{n})\overline{\chi(n)\psi(n)}\sum_{d|n}  \sum_{\substack{c_1\equiv 0(p)\\ (c_1,n/d)=1}}\psi(dc_1) \sum_{f|c_1} \mu(f)\sum_{( \lambda,c_1d)=1 }
 \chi(-f\lambda qn) \\ \int_{({s_1})=10} \int_{({s_2})=10} \widetilde{H_n}(s_1,s_2) (\frac{pq}{dc_1})^{s_1}   (\frac{pq}{dqf(-c_1r+\frac{\lambda n}{d})})^{s_2}ds_1ds_2=\\ \frac{\psi(p)\chi(-q)\sqrt{pq}}{(pq)^2}(\frac{1}{2\pi i})^3\int_{(s_1)=2} \int_{(s_2)=1/2+\delta}    \int_{(w)=1/2}   \frac{\Gamma(w)\Gamma(s_2-w)}{\Gamma(s_2)} \frac{(pq)^{s_1+s_2}}{p^{s_1+s_2-w}q^{s_2}}\\ \frac{L(\chi,w)L(\psi,s_1+s_2-w)}{L(\chi\psi,s_1+2s_2-w)} \bigg[\sum_{e=1} \frac{\mu(e)\chi(e)}{e^{w+s_1+s_2}} \sum_{d=1} \frac{1}{d^{s_1+s_2}} \sum_{a=1} \frac{\mu(a)}{a^{s_1+s_2}} \sum_{b=1}  \frac{\mu^2(b)\chi(b)}{b^{s_1+2s_2}} \\ \sum_{\substack{n \in \Z \\ n\neq 0\\ baedn \ll (pq)^\theta L^2}} \frac{\widetilde{H_{baedn}}(s_1,s_2)}{n^w} \sum_{r(baedn)} \frac{1}{r^{s_2-w}} e(\frac{d^{\circ}mr+\frac{l_1 l_2m}{d^{\circ}k^2}\overline{r}}{baedn})\overline{\psi(n)} \bigg] dwds_1ds_2. \end{multline}
\end{prop}

We point out, similar to \cite{H}, the $r$-sum on the left hand side of \eqref{eq:pp188} is almost another new Kloosterman sum, the problem is that the inner $c_2$ depends on $r.$ Further, if there were no arithmetic conditions on the $c_1,c_2$-sums in \eqref{eq:kk08} then an easy application of Mellin inversion or Poisson summation would suffice to understand how big the Kloosterman-Kloosterman term is with respect to $p,q.$ However, we must face these arithmetic difficulties. Following \cite{Y}, we use the Mellin transform \begin{equation} \frac{1}{(1+x)^s}=\frac{1}{2\pi i} \int_{(w)} \frac{\Gamma(w)\Gamma(s-w)}{\Gamma(s)} x^{-w} dw \end{equation} to analytically ``separate" the $\lambda$- sum from the $c_1$-sum.

 Now we make a contour shift that passes a pole to give a main term and a remainder term from the right hand side of \eqref{eq:pp188}. This gives \eqref{eq:pp188} equaling  \begin{multline}\label{eq:ymr88} \frac{\sqrt{pq}  \psi(p)\chi(-q)}{(pq)^2}(\frac{1}{2\pi i})^2\int_{(s_1)=2} \int_{(s_2)=1/2}       \frac{(pq)^{s_1+s_2}}{p^{s_1}q^{s_2}}\\ \frac{L(\chi,s_2)L(\psi,s_1)}{L(\chi\psi,s_1+s_2)} \bigg[\sum_{e=1} \frac{\mu(e)\chi(e)}{e^{s_1+2s_2}} \sum_{d=1} \frac{1}{d^{s_1+s_2}} \sum_{a=1} \frac{\mu(a)}{a^{s_1+s_2}} \sum_{b=1}  \frac{\mu^2(b)\chi(b)}{b^{s_1+2s_2}} \\ \sum_{\substack{n \in \Z \\ n\neq 0\\ baedn \ll (pq)^\theta L^2}} \frac{\widetilde{H_{baedn}}(s_1,s_2)}{n^{s_2}} \sum_{r(baedn)} e(\frac{d^{\circ}mr+\frac{l_1 l_2m}{d^{\circ}k^2}\overline{r}}{baedn})\overline{\psi(n)} \bigg] ds_1ds_2+\{\text{Remainder term} \eqref{eq:yrem}\}. \end{multline}  In \eqref{eq:ymr88} we clearly see Dirichlet L-functions that will aid in getting subconvexity in the Rankin-Selberg L-function.
 
 Step two of the estimation will now bring in the $m$-sum from \eqref{eq:mmain} to the main term \eqref{eq:ymr88}. First we define for $\Re(s_1)=2,\Re(s_2)=1/2$ the function \begin{multline}\label{eq:ttt88} T_{m,v}(c):=e(\frac{-mv}{c}) \int_0^\infty \int_0^\infty e(\frac{x{\dd}^2 k v}
 {\sqrt{l_1l_2}cy}+\frac{y(l_1 l_2)^{3/2}v}{{\dd}^2 k^3cx}) \\ \int_{-\infty}^{\infty}
e(\frac{tc k }{\sqrt{l_1l_2}vxy})F_{M}(t)
 V(\frac{4\pi
\sqrt{t }}{x})  W(\frac{4\pi \sqrt{t}}{y})x^{s_1-1}y^{s_2-1}\frac{dt}{t^{1/2}}dxdy.\end{multline} This is virtually the function $\widetilde{H_{\mu}}(s_1,s_2)$ with an extra exponential sum and with a change of variables $x \to \frac{\sqrt{m\dd}x}{\sqrt{pq}}$ and $y \to \frac{\sqrt{m\frac{l_1l_2}{\dd k^2}}y}{\sqrt{pq}}.$

 \begin{prop}[Step two of estimate]\label{eq:ste22}
With $p,q$ distinct primes, let $\Delta=m^2-\frac{4l_1l_2}{k^2},m \in \Z$ and $(\frac{\Delta}{\cdot})$ the quadratic character with modulus $\Delta.$ We have for any $\epsilon >0$ and any $A>0$ the asymptotic equality 

 \begin{multline}\label{eq:steb}
\frac{\sqrt{pq}\psi(p)\chi(-q)}{(pq)^2}  \sum_{m\geq 1} \frac{1}{m^{1/2}}F'_{M'}(\frac{\dd m}{pq}) \{\text{Main term of \eqref{eq:ymr88}}\}=\\  \frac{\sqrt{pq}  \psi(p)\chi(-q)}{(pq)^2}
\sum_m  (\frac{1}{2\pi i})^2 \int_{(s_1)=2} \int_{(s_2)=1/2}   (\frac{l_1l_2}{k^2})^{\frac{s_1+s_2}{2}}\int_{(w)=\epsilon}   \int_{-\infty}^\infty F'_{M'}(\frac{\dd v}{pq}) \widetilde{T_{m,v}}(w) \frac{1}{(pq)^{\frac{s_1+s_2}{2}}} \frac{(pq)^{s_1+s_2}}{p^{s_1}q^{s_2}}\\ \frac{L(s_2,\chi)L(s_1,\psi)}{L(s_1+s_2,\chi\psi)} \frac{L(s_1+2s_2+1+w,(\frac{\Delta}{\cdot})\chi)}{L(s_1+2s_2+1+w,\chi)} \frac{L(s_2+1+w,\overline{\psi})}{L(s_2+1+w,(\frac{\Delta}{\cdot})\overline{\psi}))} B(s_1+2s_2+1+w) v^{\frac{s_1+s_2-1}{2}} d v dw ds_1ds_2\\ +O((pq)^{-A}). \end{multline} Here $B(s)$ is an analytic function for $\Re(s) \geq 1$ defined in \eqref{eq:bs}.
  \end{prop}
  

This second estimate comes from the earlier reorganization where we replaced the sum $\overline{x}(c_1)$ with the sum $r(n).$ The replacement benefits us in that the $m$-sum coming from the left hand side of \eqref{eq:steb} is a sum of Kloosterman sums that is very similar to Rudinick's thesis \cite{R}: 

\begin{equation*}\label{eq:msum888}
 \sum_{m\geq 1} \frac{1}{m^{1/2}}F'_{M'}(\frac{\dd m}{pq}) S(\dd m, \frac{l_1l_2m}{\dd k^2},baedn) A(m),\end{equation*}
with $A(x)$ a nice function. This equation is analogous to our \eqref{eq:msum8}. Notice also this new Kloosterman sum has trivial nebentypus even though we initially started with forms of type $\chi(p), \psi(q).$ Applying Poisson summation modulo $baedn$ leaves us counting solutions to quadratic equations. Such a count is analogous to calculating orbital integrals in the Arthur-Selberg trace formula. The beautiful thing now, though it is not essential to the estimate, is that these solutions to quadratic equations modulo $l$ depend only the square free part of $l,$ and are also multiplicative in $l.$

 We will ultimately have the contour of the $s_1$-integral on the line $\Re(s_1)=1/2.$ This term is analogous to our \eqref{eq:res} below. Note from this term we will expect that convexity or subconvexity estimates will come into play for Dirichlet L-functions. We also make a comment that the $m$-sum in this case, which is dual to the $m$-sum from opening the norm square of the approximate functional equation, will play the role of the sum of traces of conjugacy classes in the Selberg trace formula, see \cite{R}.

To understand the third step we define $$A:=A(x,y,\dd,k,l_1,l_2):=\left(\frac{x{\dd}^2 k }
 {\sqrt{l_1l_2}y}+\frac{y(l_1 l_2)^{3/2}}{{\dd}^2 k^3x}-m\right)$$ and $$B:=B(x,y,k,l_1,l_2):=\frac{k}{\sqrt{l_1l_2}xy}.$$

Then we can rewrite the $v$-integral in \eqref{eq:steb} as  \begin{multline}\label{eq:vint0088}
(\frac{pq}{\dd})^{\frac{s_1+s_2+2w+1}{2}} \bigg[\int_0^\infty F_{M}(t)t^{\frac{s_1+s_2-1}{2}}dt\bigg] \left[\int_{-\infty}^\infty F'_{M'}(v)v^{\frac{s_1+s_2+2w-1}{2}} d v\right] \int_0^\infty \int_0^\infty  V(\frac{4\pi}{x})  W(\frac{4\pi}{y}) \\ (\frac{A}{B})^{w/2}\bigg[\frac{J_w(4\pi\sqrt{\frac{A}{B}})-J_{-w}(4\pi\sqrt{\frac{A}{B}})}{2\sin(\pi w/2)} +\frac{J_w(4\pi\sqrt{\frac{A}{B}})+J_{-w}(4\pi\sqrt{\frac{A}{B}})}{2\cos(\pi w/2)}\bigg]  x^{s_1-1}y^{s_2-1}dxdy.
\end{multline}
The steps taken to rewrite the $v$-integral in this form follow from equations \eqref{eq:vint} to \eqref{eq:res0}.

Let \begin{multline}\label{eq:dyo88}D_{\frac{A}{B}}(s_1,s_2,w):=\int_0^\infty \int_0^\infty  V(\frac{4\pi}{x})  W(\frac{4\pi}{y})\times  \\ (\frac{A}{B})^{w/2} \bigg[\frac{J_w(4\pi\sqrt{\frac{A}{B}})-J_{-w}(4\pi\sqrt{\frac{A}{B}})}{2\sin(\pi w/2)} +\frac{J_w(4\pi\sqrt{\frac{A}{B}})+J_{-w}(4\pi\sqrt{\frac{A}{B}})}{2\cos(\pi w/2)}\bigg]  x^{s_1-1}y^{s_2-1}dxdy\end{multline}

and \begin{equation}\label{eq:ccc} C(s_1,s_2,w):=\frac{L(s_2,\chi)L(s_1,\psi)}{L(s_1+s_2,\chi\psi)} \frac{L(s_1+2s_2+1+w,(\frac{\Delta}{\cdot})\chi)}{L(s_1+2s_2+1+w,\chi)} \frac{L(s_2+1+w,\overline{\psi})}{L(s_2+1+w,(\frac{\Delta}{\cdot})\overline{\psi}))} B(s_1+2s_2+1+w).\end{equation} 

Then the right hand side of step two equals  \begin{multline}\label{eq:res088} \frac{\psi(p)\chi(-q)}{pq}
\sum_m  (\frac{1}{2\pi i})^2 \int_{(s_1)=2} \int_{(s_2)=1/2}  \frac{(pq)^{s_1+s_2}}{p^{s_1}q^{s_2}}   (\frac{l_1l_2}{k^2})^{\frac{s_1+s_2}{2}}\times \\ \int_{(w)=\epsilon} (\frac{pq}{\dd})^w \bigg[\int_0^\infty F_{M}(t)t^{\frac{s_1+s_2-1}{2}}dt\bigg] \left[\int_{-\infty}^\infty F'_{M'}(v)v^{\frac{s_1+s_2+2w-1}{2}} d v\right] D_{\frac{A}{B}}(s_1,s_2,w)C(s_1,s_2,w) dw ds_1ds_2\\ +O((pq)^{-A}). \end{multline} 

From here a subconvexity estimate is needed for $C(s_1,s_2,w)$ and Bessel function analysis is required for $D_{\frac{A}{B}}(s_1,s_2,w).$ We also point out the interesting fact here that the functions $F_M,F'_{M'}$ which come from the approximate functional equation are ``unhinged" from the actual subconvexity calculation. This is good as one can always change the test functions involved in the approximate functional equation. We say though, that these integrals of $F_M$ and $F'_{M'}$ (via integration by parts) are important in ensuring the convergence of the $s_1,s_2$ and $w$-integrals.

\begin{prop}[Step three of estimate]\label{eq:hbes88}
For any $\epsilon >0$ and $A$ and $B$ defined as above, we have
\begin{itemize}
\item $$C(1/2+it_1,1/2+it_2,iy) \ll (|t_1|p)^{\theta_1} (|t_2|q)^{\theta_2} (|t_1t_2y|pq)^{\epsilon}$$ where $\theta_1,\theta_2$ are parameters $\leq 1/4.$

\item $$D_{\frac{A}{B}}(1/2+it_1,1/2+it_2,i\gamma)  \ll (t_1t_2\gamma \frac{A}{B})^{\epsilon}.$$
\end{itemize}

\end{prop}

We mention now the remainder term from \eqref{eq:ymr88} can be analyzed in the same manner as in steps two and three. We unfortunately do not get such a nice product of L-functions as in step two, but analogous analysis of that of the main term gets the same bound as the main term times an arbitrarily small but fixed power of $q.$ We write these two estimates as our final step four of the analysis of the Kloosterman-Kloosterman term.



\begin{prop}[Step four of estimate]\label{eq:fin88}
For a fixed but arbitrarily small $\delta>0,$ we have
\begin{itemize}
\item \begin{equation} \frac{\sqrt{pq}\psi(p)\chi(-q)}{(pq)^2}  \sum_{m\geq 1} \frac{1}{m^{1/2}}F'_{M'}(\frac{\dd m}{pq}) \{\text{Main term of \eqref{eq:ymr88}}\}\ll \frac{p^{\theta_1+\epsilon}q^{\theta_2+\epsilon}}{\sqrt{pq}};
\end{equation}

\item \begin{equation} \frac{ \sqrt{pq}\psi(p)\chi(-q)}{(pq)^2}  \sum_{m\geq 1} \frac{1}{m^{1/2}}F'_{M'}(\frac{\dd m}{pq}) \{\text{Remainder term of \eqref{eq:ymr88}}\}\ll \frac{p^{\theta_1+\epsilon}q^{\delta+\theta_2+\epsilon}}{\sqrt{pq}}; \end{equation}

\item \begin{equation} \{\text{Equation \eqref{eq:mmain}}\} \ll \frac{(1+q^{\delta})p^{\theta_1+\epsilon}q^{\theta_2+\epsilon}}{\sqrt{pq}}\end{equation}
\end{itemize}
where $\theta_1,\theta_2$ are parameters $\leq 1/4.$

\end{prop}

{\bf Acknowledgements.}  
I would like to thank Roman Holowinsky for suggesting this problem to me. I would also like to thank Matt Young, Paul Nelson, as well as Roman, for pointing out some errors in initial drafts of this paper.

\section{Rankin-Selberg L-function}
Let $\chi$ and $\psi$ be Dirichlet characters modulo prime numbers $p$ and $q,$ respectively, with $(p,q)=1.$ Let  $f \in H(p,\chi)$ and $g \in H(q,\psi)$ be newforms associated to $GL_2$ cuspidal automorphic representations. The Rankin-Selberg L-function is 
\begin{equation}\label{eq:RS}
L(s,f \times g):=L(\chi \psi,2s)\sum_{n =1}^\infty \frac{a_f(n)a_g(n)}{n^s}.
\end{equation}
It is known that this L-function has analytic continuation to the complex plane.
Following \cite{KMV} we can write the Rankin-Selberg L-function as a truncated Dirichlet series 
using an approximate functional equation. Namely, we can write $$L(1/2, f \times g)=\sum_{n \geq 
1} \frac{a_f(n)a_g(n)}{n^{1/2}} H(\frac{n}{pq}) + \epsilon(f \times g)\sum_{n \geq 1}  \frac{\overline
{a_f(n)a_g(n)}}{n^{1/2}} H'(\frac{n}{pq}),$$ with $|\epsilon(f \times g)|=1$ and defined in \cite{KMV}. 
The function $H$ (resp.$H'$) satisfies $H(y) \ll y^{-A},$ for all $A>0.$ As well for $y$ small  $H$ (resp.$H'$) satisfies $$H(y) \ll  (pq)^{\epsilon}|\log y|.$$ We will also note that in estimating from line to line $\{expression\}^{\epsilon}$ will note bounding by an arbitrarily small positive power, and is not necessarily the same $\epsilon$ line to line. 

\section{Kuznetsov trace formula}

        We start by defining the Kuznetsov trace formula. We refer to  \cite{Iw} book on
    it's derivation. Let $S(\Gamma_0(p),\chi)$ be the space of
    holomorphic cusp forms of weight $k$ for the group $\Gamma_0(p).$ For each
    form $\phi \in S_k(\Gamma_0(p),\chi),$ let $c_n(\phi)$ be the $n$-th
    Fourier coefficient, then define $$a_n(\phi):=
    \sqrt{\frac{\pi^{-k} \Gamma(k)} {(4n)^{k-1}}}c_n(\phi).$$
Likewise, for Maass cusp forms we define
        $$a_n(\phi):=(\frac{4\pi |n|}{\cosh(\pi
        s)})^{1/2}\rho(n),$$ where $\phi$ has $L^2$ norm one and
        eigenvalue $1/4+s^2$ with Fourier expansion $$
        \phi(z)=\sum_{n \neq 0} \rho(n)W_s(nz).$$ Here $W_s(nz)=2\sqrt{y}
        K_{s-1/2}(ny)e(x).$
 The continuous spectrum coefficients are defined as $$\eta(l,1/2+it):=2\pi^{1+it}{\cosh(\pi t)}^{-1/2} \frac{\tau_{it}(n)}{\Gamma(1/2+it)\zeta(1+2it)},$$ 
where $\tau_{it}(n)=\sum_{ab=n}(a/b)^{it}.$

    Let $V \in C^{\infty}_0(\R^{+}),$ then the Kuznetsov formula states
\begin{equation}\label{eq:kuznet}\sum_{\phi}
h(V,\lambda_{\phi})a_n(\phi)\overline{a_l(\phi)}+\frac{1}{4\pi}
\int_{-\infty}^\infty h(V,t)\eta(n,1/2
+it)\overline{\eta(l,1/2+it)}dt\end{equation} $$= \frac{\delta_{n,l}}{\pi}\left[ \int_{-\infty}^\infty h(V,t) \text{tanh}(\pi t)
t dt  \quad  +  \sum_{k>0,k\text{ even}}
(k-1)h(V,k) \right] +
\sum_{c\equiv 0(p)}^\infty \frac{1}{c} S_{\chi}(l,n,c)V(4\pi \sqrt{ln}/c),$$
 where the sum $\phi$ is
over an orthonormal basis for $S_k(\Gamma_0), k \in 2 \Z$ and
Maass forms w.r.t. the Petersson inner product, and $V \in
C_0^{\infty}(\R-\{0\}).$ The Kloosterman sum is $$S_{\chi}(l,n,c)=\sum_{x(c)^{*}}\chi(x)e(\frac{lx+n\overline{x}}{c}),$$ as well \begin{equation}
 h(V,\lambda)
 := \left\{ \begin{array}{ll}
         i^{k}\int_0^\infty V(x)J_{\lambda -1}(x)x^{-1}dx & \text{if } \lambda  \in  2 \Z; \medskip \\
         \int_0^\infty
V(x)B_{2i\lambda}(x)x^{-1}dx & \text{if } \lambda \in
\R-2\Z.\end{array} \right.  \end{equation} Here, $B_{2it}(x) = (
\text{2 sin}(\pi it))^{-1}(J_{-2it}(x) - J_{2it}(x)),$ where
$J_\mu(x)$ is the standard $J$-Bessel function of index $\mu$ (See
\cite{IK} and \cite{Wat}).

The delta (main) term of the geometric side of the trace formula has associated to it \begin{equation*} \delta_{a,b}
 := \left\{ \begin{array}{ll}
         1 & \text{if } a=b,  \medskip \\
       0 & \text{if } a\neq b. \end{array} \right. \end{equation*}

\section{Including an Amplifier}

 We build in an amplifier, and thus our starting point is \begin{equation}\label{eq:beg}\sum_{f\in H(p,\chi)}h(V,t_f)|\sum_{l \leq L} x_l \lambda_f(l)|^2 \sum_{h \in H(q,\psi) }h(W,t_h) |L(1/2,f \times h)|^2,\end{equation} where $x_l$ are arbitrary complex coefficients.  Specifically, $|L(1/2,f \times h)|^2$ from the approximate functional equation above is equal to four sums given by \begin{equation}\label{eq:approx}
\sum_{f\in H(p,\chi),h \in H(q,\psi) }h(V,t_f)h(W,t_h) \sum_{m,n \geq 1} \frac{1}{(mn)^{1/2}} H(\frac{n}{pq})H'(\frac{m}{pq}) a_f(n)a_f(m)a_h(n)a_h(m).
\end{equation} up to a root number $\epsilon(f \times h).$ We now introduce a partition of unity for the $n$- and $m$-sums similar to \cite{KMV}. Let $\eta(x)$ be a smooth function which is zero for $x\leq 1/2$ and one for $x\geq 1 $ such that $$\eta(x)=\sum_{M\geq 1} \eta_{M}(x)$$ with $\eta_M$ compactly supported in $[M/2,2M]$ such that $x^i \eta_M(x)^{(i)} \ll_i 1$ for any $i \geq 0.$ The number of partitions $M$ less than $X$ is $O(\log X).$ We use such a partition on the two sums $n,m,$ and define $$F(x):=H(x)\eta(pqx)=\sum_{M \geq 1} H(x)\eta_M(pqx):=\sum_{M\geq 1} F_M(x).$$ Following the analysis of \cite{KMV} we also have the bound for all $i,A \geq 0$  \begin{equation}\label{eq:pest}
x^i \frac{\partial^i}{\partial^i x} F_M(x) \ll_{A,i} \frac{(pq)^{\epsilon}}{M^{1/2}}(\frac{pq}{x})^A
\end{equation} for any $\epsilon >0.$
The notation for the analogous partition on the $m$-sum and $H'$ will be $\eta_{M'},F'_{M'}.$ Also note by the decay in $H$ (resp. $H'$) in the definition of the approximate functional equation we have $M \leq (pq)^{1+\epsilon}.$

So we will study the sum  
\begin{multline}\label{eq:approx007}
\sum_{M,M' \geq 1} \sum_{f\in H(p,\chi),h \in H(q,\psi) }h(V,t_f)h(W,t_h) |\sum_{l \leq L} x_l \lambda_f(l)|^2  \sum_{m,n \geq 1} \frac{1}{(mn)^{1/2}} F_M(\frac{n}{pq})F'_{M'}(\frac{m}{pq})\times \\ a_f(n)a_f(m)a_h(n)a_h(m).
\end{multline} 
Now opening the square and using Hecke properties, we get
$$ |\sum_{l \leq L} x_l \lambda_f(l)|^2=\sum_{l_1,l_2 \leq L} x_{l_1} x_{l_2} \sum_{k|(l_1,l_2)}\lambda_f(\frac{l_1 l_2}{k^2}).$$ Incorporating the above equation into \eqref{eq:approx007} for a fixed $M,M'$ from our partition of unity  we have \begin{multline}\label{eq:approxie} \sum_{l_1,l_2 \leq L} x_{l_1} x_{l_2}\sum_{k|(l_1,l_2)}  \sum_{\dd|{\frac{l_1l_2}{k^2}}} \frac{1}{{\dd}^{1/2}} \sum_{m \geq 1} \frac{1}{m^{1/2}}F'_{M'}(\frac{\dd m}{pq}) \\ \left[  \sum_{n \geq 1} \frac{1}{n^{1/2}} F_M(\frac{n}{pq})   \sum_{f\in H(p,\chi),h \in H(q,\psi) }h(V,t_f)h(W,t_h) a_f(n)a_h(\dd m)a_h(n)a_f(\frac{l_1 l_2}{\dd k^2} m)\right].
\end{multline}
Let us define $$D_{n,l}(V):= \frac{\delta_{n,l}}{\pi}\left[ \int_{-\infty}^\infty h(V,t) \text{tanh}(\pi t)
t dt  \quad  +  \sum_{k>0,k\text{ even}}
(k-1)h(V,k) \right].$$ Then performing Kuznetsov twice and Poisson summation on the $n$-sum in the $\text{\{Kloosterman-Kloosterman term \}}$ we have, as in section 7 of \cite{H}, 

\begin{multline}\label{eq:poisson}    \sum_{l_1,l_2 \leq L} x_{l_1} x_{l_2}\sum_{q|(l_1,l_2)}  \sum_{\dd|{\frac{l_1l_2}{k^2}}} \frac{1}{{\dd}^{1/2}} \sum_{m \geq 1} \frac{1}{m^{1/2}}F'_{M'}(\frac{\dd m}{pq})\times \\ \vast\{ \sum_{n \geq 1} \frac{1}{n^{1/2}} F_M(\frac{n}{pq}) D_{n,\frac{l_1 l_2m}{\dd k^2}}(V)D_{n,\dd m}(W)  \text{ \{delta-delta term\}}+ \\  \sum_{n \geq 1} \frac{1}{n^{1/2}} F_M(\frac{n}{pq}) D_{n,\frac{l_1 l_2m}{\dd k^2}}(V) \bigg(\sum_{c_2\equiv 0(q)} \frac{S_{\psi}(n,\dd m,c_2)}{c_2}W(\frac{4\pi \sqrt{n \dd m}}{c_2})\bigg)\text{ \{delta-Kloosterman term\}} +\\
 \sum_{n \geq 1} \frac{1}{n^{1/2}} F_M(\frac{n}{pq}) D_{n,\dd m}(W)\bigg(\sum_{c_1\equiv 0(p)} \frac{S_{\chi}(n,\frac{l_1 l_2m}{\dd k^2},c_1)}{c_1}V(\frac{4\pi \sqrt{ \frac{nl_1 l_2m}{\dd k^2}}}{c_1})\bigg)\text{ \{Kloosterman-delta term\}}+ \\
\sqrt{pq} \sum_{\substack{c_1\equiv 0(p)\\c_2\equiv 0(q)}} \frac{1}{c_1 c_2}\sum_{x  (c_1)^{*}}
 \sum_{y  (c_2)^{*}} \chi(\overline{x})\psi(\overline{y})e(\frac{\overline{x}\dd m}{c_1}+\frac{\overline{y}\frac{l_1 l_2m}{\dd k^2}}{c_2}) \times \\ \bigg\{
 \sum_k
 \int_{-\infty}^{\infty} e(pqt(\frac{c_2 x + c_1 y -kc_1c_2}{c_1 c_2})) \frac{F_M(t)}{t^{1/2}}
 V(\frac{4\pi
\sqrt{tpq\dd m}}{c_1})  W(\frac{4\pi \sqrt{tpq\frac{l_1 l_2m}{\dd k^2}}}{c_2})dt\bigg\} \\ \text{ \{Kloosterman-Kloosterman term\}}\vast\}.  \end{multline}

We will deal with the hardest case first (Kloosterman-Kloosterman term) in Section \ref{zero} and Section \ref{remb}, then in Section \ref{dd} we deal with the delta-delta term. In Section \ref{dk} following that, we deal with both the delta-Kloosterman term and Kloosterman-delta term.


\section{Kloosterman-Kloosterman term}\label{zero}

We let $n= c_2 x + c_1 y -kc_1c_2,$ and reparametrize the Kloosterman-Kloosterman term of \eqref{eq:poisson} analogous to the equations up to (7.6) of \cite{H}, quoting the propositions we need from that paper. The only difference between lemmas we need here and that of \cite{H} is that, $c_1\equiv 0(p)$ and $c_2 \equiv 0(q)$ as there we have only full level. This does not affect the proofs of this reparametrization and so we state the propositions we need from \cite{H}.
\begin{definition}\label{equiv}
Let $\overline{X}(c_1,c_2,n)$ denote the equivalence classes of
pairs $(x,y)$ with $x,y\in\mathbf{Z}$ such that $(x,c_1)=1$,
$(y,c_2)=1$, and
$$
c_2x+c_1y = n.
$$
Here we say that $(x,y)$ is equivalent to $(x',y')$ if $x\equiv
x'\pmod{c_1}$ and $y\equiv y'\pmod{c_2}$. Let $X(c_1,c_2,n)$ be a
set of representatives for the classes in
$\overline{X}(c_1,c_2,n)$.
\end{definition}

\begin{prop}\label{n00}
Let $(x,y)\in X(c_1,c_2,0),$ then $x=-y, c_1=c_2.$ This implies $c_1=c_2\equiv 0(pq).$
\end{prop}

\begin{prop}\label{inverse} Let $(x,y)\in X(c_1,c_2,n)$ and $\xbar \in\mathbf{Z} $ be an inverse of $x$ modulo $c_1$ and $\ybar \in\mathbf{Z}$ be an inverse of $y$ modulo
$c_2$. Then there exists a pair $(r_1,r_2)$ such that
$r_1r_2\equiv 1\pmod{n}$ and
\begin{equation}\label{a}
\xbar = \frac{c_2+c_1r_1}{n},\qquad \ybar = \frac{c_1+c_2r_2}{n}
\end{equation}
The pair $(r_1,r_2)$ is uniquely determined modulo $n$ by the
equivalence class of the pair $(x,y)$, and the map from
$X(c_1,c_2,n)$ to the set of pairs $(r_1,r_2)$ modulo $n$ is
injective.
\end{prop}

 \begin{definition} Let $c_1$, $c_2$ be positive integers. Set  $d = (c_1,c_2)$. Assume that $d|n$.
 Let $Y(c_1,c_2,n)$ be the set of  classes $r\in(\Z/n)^*$ such that
\begin{enumerate}
\item[(a')] $(c_1/d)r+(c_2/d) \equiv 0\pmod{\frac{n}{d}}$
\item[(b')] $(c_1/d)r+(c_2/d) \not\equiv 0 \pmod{\frac{n}{d'}}$
if  $d'|d$ and  $d'<d$.
\end{enumerate}
 \end{definition}

 \begin{prop}\label{bijection} The map $i:(x,y)\to r_1$ defines a bijection between $X(c_1,c_2,n)$
 and $Y(c_1, c_2,n)$.
 \end{prop} For proofs of Propositions \ref{n00}, \ref{inverse} and \ref{bijection},
see \cite{H}.

Now we break the Kloosterman-Kloosterman term into $2$ cases: $n=0$ and $n \neq 0.$

\subsection{ [Case $n=0$]}
 This is by far the easier of the two cases. Using Proposition \ref{n00} above for $n=0,$ we must have $c_1=c_2\equiv 0(pq)$ and $\overline{x}=-\overline{y}.$

\begin{equation}\label{eq:poisson00} \sqrt{pq} \sum_{m\geq 1} \frac{1}{m^{1/2}}F'_{M'}(\frac{\dd m}{pq}) \sum_{c_1\equiv 0(qp)} \frac{1}{c_1^2}\sum_{x  (c_1)^{*}}
 \chi(\overline{x})\psi(-\overline{x})e(\frac{\overline{x}m(\dd -\frac{l_1 l_2}{\dd k^2})}{c_1})\times  \end{equation}
 $$\int_{-\infty}^{\infty} \frac{F_{M}(t)}{t^{1/2}}
 V(\frac{4\pi
\sqrt{tpq\dd m}}{c_1})  W(\frac{4\pi \sqrt{tpq\frac{l_1 l_2m}{\dd k^2}}}{c_1})dt.$$

Note if $\dd -\frac{l_1 l_2}{\dd k^2}=0,$ then character sum is zero as $\psi \neq \chi.$ So we assume for now on $\dd -\frac{l_1 l_2}{\dd k^2}\neq 0.$

Now the $x$-sum above is a Gauss sum with primitive character $\chi \psi$ of conductor $pq.$ In order for this sum to be non-zero we must have $$e(\frac{\overline{x}m(\dd -\frac{l_1 l_2}{\dd k^2})}{c_1})=e(\frac{\overline{x}m'a}{pq}),$$ where $m'|m, a|(\dd -\frac{l_1 l_2}{\dd k^2}),$ and $(m'a,pq)=1.$ Further we must have $c_1=pqr,$ where $(r,pq)=1$ otherwise the whole term \eqref{eq:poisson00} is zero as the support of $V,W, $ and $g$ implies $m \sim pq$ and so $r$ must be independent of $p,q.$ 

Then by Chinese remainder theorem the Gauss sum equals \begin{multline}\label{eq:ggg}
\sum_{x  (c_1)^{*}}
 \chi(\overline{x})\psi(-\overline{x})e(\frac{\overline{x}m(\dd -\frac{l_1 l_2}{\dd k^2})}{c_1})=\bigg[\sum_{x  (p)^{*}}
 \chi(\overline{x})e(\frac{\overline{qrx}m(\dd -\frac{l_1 l_2}{\dd k^2})}{p})\bigg]\times \\ \bigg[\sum_{y  (q)^{*}}
 \psi(\overline{y})e(\frac{\overline{pry}m(\dd -\frac{l_1 l_2}{\dd k^2})}{q})\bigg] \bigg[\sum_{z  (r)^{*}}
e(\frac{\overline{pqz}m(\dd -\frac{l_1 l_2}{\dd k^2})}{r})\bigg].
\end{multline}

Now to get a non-zero contribution for \eqref{eq:poisson00}, $$(m(\dd -\frac{l_1 l_2}{\dd k^2}),pq)=1.$$

With an easy change of variables, \eqref{eq:ggg} equals $$\overline{\chi(qr)}\overline{\psi(pr)}\overline{\chi(m(\dd -\frac{l_1 l_2}{\dd k^2})})  \overline{\psi(m(\dd -\frac{l_1 l_2}{\dd k^2})}) \tau(\chi)\tau(\psi) f_{\frac{c_1}{pq}}(m(\dd -\frac{l_1 l_2}{\dd k^2})),$$

where $$f_n(m)=\sum_{b|(m,n)}\mu(\frac{n}{b})b,$$ $\tau(\chi),\tau(\psi)$ are Gauss sums, and $r=\frac{c_1}{pq}.$

Let \begin{equation} Z(x):=\frac{1}{x}\int_{-\infty}^{\infty}
\frac{F_{M}(t)}{t^{1/2}}V(\frac{4\pi \sqrt{t \dd m}}{x})W(\frac{4\pi
\sqrt{\frac{l_1 l_2m}{\dd k^2}t}}{x})dt.\end{equation} Denoting the Mellin transform
of $Z(x)$ as $ \widetilde{Z}(s)$, we use Mellin inversion to write the $c_1$-sum in \eqref{eq:poisson00} as
\begin{equation}\label{eq:ja}
 \frac{1}{\sqrt{pq}} \sum_{c_1\equiv 0(pq)}
\frac{f_{\frac{c_1}{pq}}(m(\dd -\frac{l_1 l_2}{\dd k^2}))}{c_1} Z(\frac{c_1}{\sqrt{pq}})=
    \left\{ \frac{1}{2\pi
i}\int_{(\sigma)}
\widetilde{Z}(s)L(s)(\sqrt{pq})^{s} ds \right\} \end{equation} for $\sigma>0$ large, and

where $$L(s):=\sum_{c_1\equiv 0(pq)}^\infty \frac{f_{\frac{c_1}{pq}}(m(\dd -\frac{l_1 l_2}{\dd k^2}))}{c_1^{1+s}}=\frac{1}{(pq)^{1+s}}\left[ \sum_{b|m(\dd -\frac{l_1 l_2}{\dd k^2})}\frac{1}{b^s}\right] \sum_{c_1=1}^\infty \frac{\mu(c_1)}{c_1^{1+s}}.$$
The second equality follows from inverting the $c_1$- and $b$-sum.

 In \eqref{eq:ja} we can shift the contour again to $\Re(s)=1$ bounding the $b$-sum by $(m(\dd -\frac{l_1 l_2}{\dd k^2}))^{\epsilon}.$ We have left to estimate \begin{equation}\label{eq:myo}
\frac{\tau(\chi)\tau(\psi)}{(pq)^{3/2}} \sum_{m\geq 1} \frac{\chi(m)\psi(m)m^{\epsilon}}{m^{1/2}}F'_{M'}(\frac{\dd m}{pq}) \int_{(1)} \left[\int_{-\infty}^{\infty}
\frac{F_{M}(t)}{t^{1/2}}V(\frac{4\pi \sqrt{t \dd m}}{x})W(\frac{4\pi
\sqrt{\frac{l_1 l_2m}{\dd k^2}t}}{x})dt x^{s-2}dx\right] ds.
\end{equation}
 
Let $$Z'(z):=  F'_{M'}(\frac{\dd z}{pq}) \int_{(1)}\left[ \int_{-\infty}^{\infty}
\frac{F_{M}(t)}{t^{1/2}}V(\frac{4\pi \sqrt{t \dd z}}{x})W(\frac{4\pi
\sqrt{\frac{t l_1 l_2 z}{\dd k^2}}}{x})dt x^{s-2}dx\right] ds.$$

Then \eqref{eq:myo} equals $$\frac{\tau(\chi)\tau(\psi)}{(pq)^{3/2}}\frac{1}{2\pi i} \int_{(\sigma_1)} \widetilde{Z'}(w) L(1/2+w-\epsilon,\chi \psi) (pq)^w dw.$$ We can shift this contour to $\Re(w)=\epsilon$ where $\epsilon >0$ without hitting a pole. Then we note $$\widetilde{Z'}(w)=\int_0^\infty F'_{M'}(\frac{\dd v}{pq})  \left[\int_{(1)} \int_{-\infty}^{\infty}
\frac{F_{M}(t)}{t^{1/2}}V(\frac{4\pi \sqrt{t \dd v}}{x})W(\frac{4\pi
\sqrt{\frac{t l_1 l_2 v}{\dd k^2}}}{x})dt x^{s-2}dx\right] ds v^{w-1} dv \ll (pq)^{\epsilon},$$ and also by the support of the functions $V,W,$ we get $\dd k\sim \sqrt{l_1 l_2}.$

The convexity bound in the level aspect for a Dirichlet L-function, with the level $q,$ is $$L(1/2+it,\chi) \ll q^{1/4+\epsilon}.$$ A subconvexity bound is $$L(1/2+it,\chi) \ll q^{\theta_1+\epsilon},$$ where  $\theta_1 <1/4.$  Not to discriminate using a convexity or subconvexity bound in this paper, we call it a (sub)convexity bound and note \eqref{eq:myo} is bounded by $$(pq)^{\theta_1-1+\epsilon}.$$
 
Therefore using the fact that $\dd q\sim \sqrt{l_1 l_2},$ with the sum over $l_1,l_2,$ we have  \begin{equation}\label{eq:zerothe}(pq)^{\theta_1-1+\epsilon}  \sum_{l_1,l_2 \leq L}\frac{ x_{l_1} x_{l_2}}{(l_1l_2)^{1/4}}\sum_{q|(l_1,l_2)} \sqrt{q} \ll (pq)^{\theta_1-1+\epsilon} L^{\epsilon} ||x||_2^2.\end{equation}

\subsection{ [Case $n\neq0$]}

\subsubsection{Explanation of this case}
Let us summarize what happens in this case, as it is the most complicated case in the paper. After using the bijection of \cite{H}, we have reparametrized the $c_1,c_2,$ and $k$- sums in \eqref{eq:poisson} into a new $n$-sum. This puts arithmetic conditions on the $c_1,c_2$-sums that we remove by standard techniques (Mobius inversion, inverting sums). This leaves us with the result of Proposition \ref{bigp}. From there we apply Poisson summation in the $m$-sum in Section \ref{myoy} coming from \eqref{eq:klokl}. This lets us exchange Kloosterman sums for counting solutions to quadratic equations as in the thesis of Rudnick \cite{R}. This same Poisson summation in \cite{R} gives entry to the geometric side of the Selberg trace formula. Now we can write the leftover sums coming from the previous Mobius inversion and Poisson summation in terms of a product of various L-functions. The sizes of these L-functions will directly control how large the final estimation is for \eqref{eq:beg}. The expression in terms of L-functions is seen in \eqref{eq:res}. From there we require bounding the archimedean integrals that arise in the calculation. After several changes of variables it is evident that the important archimedean estimation comes from an unexpected Bessel function that  arises. Estimating these integrals, including the Bessel function, will only depend on the amplifier parameter $L.$ We see that specific dependence only on the amplifier in the analysis of Proposition \ref{dyo}. Then we include the terms from the amplifier itself and do a final estimation.


\subsubsection{Analysis of $n \neq 0$ case}
Using the bijection we get $\overline{x}=\frac{c_1 r+c_2}{n}(c_1),$ and $\chi(\overline{x})=\chi(c_2)\chi(\overline{n})$ with certain conditions on the $c_1,c_2$-sums which are explicated in section 8 of \cite{H}.

Let \begin{equation}
\label{eq:iii} I(n,x,y): = \int_{-\infty}^{\infty}
e(\frac{tn}{xy})F_{M}(t)
 V(\frac{4\pi
\sqrt{\frac{td^{\circ}m}{pq}}}{x})  W(\frac{4\pi \sqrt{\frac{t_1 l_1l_2m}{d^{\circ}k^2pq}}}{y})\frac{dt}{t^{1/2}},
\end{equation}

then following \cite{H} exactly we have
\begin{multline}\label{eq:klokl}
 \text{\{Kloo.-Kloo. term\}}=\sqrt{pq}  \sum_{n\in\mathbf{Z}} \sum_{m\geq 1} \frac{1}{m^{1/2}}F'_{M'}(\frac{\dd m}{pq}) \sum_{\substack{c_1\equiv 0(p)\\c_2\equiv 0(q)}} \frac{1}{c_1 c_2}\\ \sum_{(x,y)\in X(c_1,c_2,n)}
\chi(\overline{x})\psi(\overline{y}) e(\frac{\overline{x}l}{c_1}+\frac{\overline{y}l'}{c_2})  I(n,\frac{c_1}{pq},\frac{c_2}{pq})
\end{multline}

Following the argument of section 8 of \cite{H} closely, the $\text{\{Kloosterman-Kloosterman term\}} $ (without the $n=0$ term from last section) equals 
\begin{multline}\label{eq:kk} \frac{\sqrt{pq}}{(pq)^2} \sum_{m\geq 1} \frac{1}{m^{1/2}}F'_{M'}(\frac{d^{\circ}m}{pq})  \sum_{n\neq0 \in \Z} \sum_{r(n)} e(\frac{d^{\circ}mr+\frac{l_1 l_2m}{d^{\circ}k^2}\overline{r}}{n})\overline{\chi(n)\psi(n)}\sum_{d|n} \\  \sum_{\substack{c_1\equiv 0(p)\\ (c_1,n/d)=1}} \sum_{ \substack{ c_2 = -(c_1\,r + \lambda\frac{n}d), c_2\equiv 0(q)\\(\lambda,c_1d)=1 }} \chi(dc_2)\psi(dc_1)H_n(\frac{dc_1}{pq},
\frac{dc_2}{pq}),\end{multline}
 where  \begin{equation} \label{eq:defhn}H_n(x,
y):=\frac{1}{xy}e(\frac{xdm}{ny}+\frac{y\frac{l_1 l_2m}{dk^2}}{nx})I(n,x,y).\end{equation}

This is equivalent to fixing $n,m,$ and $d$ from \eqref{eq:kk} and looking at \begin{equation}\label{eq:mc1} \sum_{\substack{c_1\equiv 0(p)\\ (c_1,n/d)=1}}\psi(dc_1) \sum_{ \substack{ c_2 = -c_1\,r + \lambda\frac{n}d, c_2\equiv 0(q)\\ (c_2,c_1)=1\\(\lambda,d)=1 }}
 \chi(dc_2) H_n(\frac{dc_1}{pq}.
\frac{dc_2}{pq}).
\end{equation}
We apply Mobius inversion on the condition $(c_1,c_2)=1$ to get \begin{equation}\label{eq:mc2}   \sum_{\substack{c_1\equiv 0(p)\\ (c_1,n/d)=1}}\psi(dc_1) \sum_{f|c_1} \mu(f)\sum_{ \substack{ c_2 = -f(c_1\,r + \lambda\frac{n}d), fc_2\equiv 0(q)\\(\lambda,d)=1 }}
 \chi(dfc_2) H_n(\frac{dc_1}{pq},
\frac{dfc_2}{pq}).
\end{equation}

Using the fact that $\chi$ is a character modulo $p$ and $p|c_1,$ $\chi$ will only depend on $-f\lambda n,$ leaving
\begin{equation}\label{eq:mc3}   \sum_{\substack{c_1\equiv 0(p)\\ (c_1,n/d)=1}}\psi(dc_1) \sum_{f|c_1} \mu(f) \sum_{ (\lambda,d)=1 }
 \chi(-f\lambda qn) H_n(\frac{dc_1}{pq},
\frac{dqf(-c_1r+\frac{\lambda n}{d})}{pq}).
\end{equation}

We will be interchanging many sums and integrals so we require an estimate that allows to interchange the $n$-sum.
\begin{definition}
A function $F$ of $p,q$ will be called negligible if for any $A>0$ we have $$F(p,q) \ll (pq)^{-A}.$$

\end{definition}

\begin{lemma}\label{ncc}
The $n$-sum can be limited to length $O((pq)^{\theta} L^2)$ for a small but fixed $\theta >0,$ while the complementary sum is negligible.

\end{lemma}

\begin{proof}

By integration by parts on the integral $I(n,x,y)$ we get
\begin{multline}\label{eq:blobo}\sum_{i+j+k=1} \frac{xy}{2\pi i Xn}  \int_{-\infty}^{\infty}
e(\frac{tn}{xy})F_M^{(k)}(t) \left((\frac{2\pi \sqrt{\frac{d^{\circ}m}{pqt}}}{x})^{i}V^{(i)}(\frac{4\pi \sqrt{\frac{td^{\circ}m}{pq}}}{x})\right) \times \\ \left(\frac{2\pi \sqrt{\frac{t_1 l_1l_2m}{d^{\circ}k^2pqt}}}{y})^j W^{j}(\frac{4\pi \sqrt{t\frac{l_1 l_2m}{pq \dd k^2}}}{y})\right) dt. \end{multline} Notice by the compact support of $V,W,$ and $F'_{M'}$ we have $ \dd m \sim pq, t=O(1),$ so $x$ is bounded and $y \sim \frac{l_1l_2}{{\dd}^2 k^2}.$ So we have $\frac{xy}{ n}\ll \frac{L^2}{n}.$ The functions $V,W, F_M$ are smooth and independent of $p,q$ except for the function $F_M(x).$  One gets from \eqref{eq:pest} the estimate $F^{(i)}_M(x) \ll (pq)^{\epsilon}$ for any $\epsilon>0.$  Assuming $n\gg (pq)^{\theta} L^2$ then by the above analysis \eqref{eq:blobo} is bounded by $O((pq)^{\epsilon-\theta}).$ We choose $\epsilon < \theta$ so that the integral decays in $p$ and $q.$ Then integration by parts $Q$ times gives the bound for \eqref{eq:blobo} of $O((pq)^{-Q\theta+\epsilon}).$ For $Q$ large enough and the $n$-sum in this range, \eqref{eq:iii} is negligible.
\end{proof}

We now can limit the $n$-sum to a length $O((pq)^\theta L^2)$ for a fixed but small $\theta>0.$ Using the lemma and including the $n$- and $d$-sums into \eqref{eq:mc3} we now prove

\begin{prop}\label{bigp}
Let $\delta,\theta>0$ be small but fixed. We have \begin{multline} \label{eq:pp1}\frac{\sqrt{pq}}{(pq)^2} \sum_{\substack{n \in \Z\\n\neq 0\\ n \ll (pq)^\theta L^2}} \sum_{r(n)} e(\frac{d^{\circ}mr+\frac{l_1 l_2m}{{\dd}^{\circ}k^2}\overline{r}}{n})\overline{\chi(n)\psi(n)}\sum_{d|n}  \sum_{\substack{c_1\equiv 0(p)\\ (c_1,n/d)=1}}\psi(dc_1) \sum_{f|c_1} \mu(f)\sum_{( \lambda,d)=1 }
 \chi(-f\lambda qn) \\ \int_{({s_1})=10} \int_{({s_2})=10} \widetilde{H_n}(s_1,s_2) (\frac{pq}{dc_1})^{s_1}   (\frac{pq}{dqf(-c_1r+\frac{\lambda n}{d})})^{s_2}ds_1ds_2=\\ \frac{\psi(p)\chi(-q)\sqrt{pq}}{(pq)^2}(\frac{1}{2\pi i})^3\int_{(s_1)=2} \int_{(s_2)=1/2+\delta}    \int_{(w)=1/2}   \frac{\Gamma(w)\Gamma(s_2-w)}{\Gamma(s_2)} \frac{(pq)^{s_1+s_2}}{p^{s_1+s_2-w}q^{s_2}}\\ \frac{L(\chi,w)L(\psi,s_1+s_2-w)}{L(\chi\psi,s_1+2s_2-w)} \bigg[\sum_{e=1} \frac{\mu(e)\chi(e)}{e^{w+s_1+s_2}} \sum_{d=1} \frac{1}{d^{s_1+s_2}} \sum_{a=1} \frac{\mu(a)}{a^{s_1+s_2}} \sum_{b=1}  \frac{\mu^2(b)\chi(b)}{b^{s_1+2s_2}} \\ \sum_{\substack{n \in \Z \\ n\neq 0\\ baedn \ll (pq)^\theta L^2}} \frac{\widetilde{H_{baedn}}(s_1,s_2)}{n^w} \sum_{r(baedn)} \frac{1}{r^{s_2-w}} e(\frac{d^{\circ}mr+\frac{l_1 l_2m}{d^{\circ}k^2}\overline{r}}{baedn})\overline{\psi(n)} \bigg] dwds_1ds_2. \end{multline}

\end{prop}

\begin{proof}

Incorporating the $n$ and $d$ sum back into \eqref{eq:mc3}, and applying Mellin inversion to the $H_n$ function we have \begin{multline}\label{eq:ym}
\frac{\sqrt{pq}}{(pq)^2} \sum_{\substack{n \in \Z\\n\neq 0\\ n \ll (pq)^\theta L^2}} \sum_{r(n)} e(\frac{d^{\circ}mr+\frac{l_1 l_2m}{d^{\circ}k^2}\overline{r}}{n})\overline{\chi(n)\psi(n)}\sum_{d|n}  \sum_{\substack{c_1\equiv 0(p)\\ (c_1,n/d)=1}}\psi(dc_1) \sum_{f|c_1} \mu(f)\sum_{( \lambda,d)=1 }
 \chi(-f\lambda qn) \\ \int_{({s_1})} \int_{({s_2})} \widetilde{H_n}(s_1,s_2) (\frac{pq}{dc_1})^{s_1}   (\frac{pq}{dqf(-c_1r+\frac{\lambda n}{d})})^{s_2}ds_1ds_2.
\end{multline}

We need to arithmetically separate the $c_1$- and $\lambda$-sum. Following \cite{Y}, we use the Mellin transform \begin{equation}
\frac{1}{(1+x)^s}=\frac{1}{2\pi i} \int_{(w)} \frac{\Gamma(w)\Gamma(s-w)}{\Gamma(s)} x^{-w} dw
\end{equation} with $\Re(w) \leq \Re(s),$ to detect the condition $$\left(\frac{1}{-c_1rdqf(1+\frac{\lambda n}{-c_1rd})}\right)^{s_2}.$$

So the integral can be rewritten as $$(\frac{1}{2\pi i})^3\int_{({s_1})} \int_{({s_2})}  \int_{(w)}  \widetilde{H_n}(s_1,s_2) \frac{\Gamma(w)\Gamma(s_2-w)}{\Gamma(s_2)}  (\frac{pq}{dc_1})^{s_1}  (\frac{pq}{-c_1rdqf})^{s_2} (\frac{-c_1rd}{\lambda n})^{w}dw ds_1ds_2,$$ with $\Re(s_1)=2,$ $\Re(s_2)=1/2+\delta,$ and $\Re(w)=1/2$ with $\delta>0$ small.

As the $n,c_1,\lambda$-sums are bounded we can interchange them with the integrals. The next goal is to then is detect the many relativly prime conditions with Mobius inversion and rewrite the sums in terms of Dirichlet L-functions. We do this in several steps for clarity. Fix $n,d,$ and $r$ for now. Apply Mobius inversion to the condition $(\lambda,d)=1$ to get the $\lambda$-sum equal to $$\sum_{e|d} \frac{\mu(e)\chi(e)}{e^w} \sum_{\lambda=1} \frac{\chi(\lambda)}{\lambda^{w}}.$$

Let us collect terms now from \eqref{eq:ym},  \begin{multline}\label{eq:ym1}
\frac{\sqrt{pq}}{(pq)^2} \chi(-q)(\frac{1}{2\pi i})^3\int_{(s_1)} \int_{(s_2)}  \int_{(w)}   \frac{\Gamma(w)\Gamma(s_2-w)}{\Gamma(s_2)} \frac{(pq)^{s_1+s_2}}{q^{s_2}} \\ \bigg[L(\chi,w) \sum_{\substack{n \in \Z \\ n\neq 0\\ n \ll (pq)^\theta L^2}}\frac{\widetilde{H_n}(s_1,s_2)}{n^w} \sum_{r(n)} \frac{1}{r^{s_2-w}} e(\frac{d^{\circ}mr+\frac{l_1 l_2m}{d^{\circ}k^2}\overline{r}}{n})\overline{\psi(n)} \sum_{d|n} \frac{\psi(d)}{d^{s_1+s_2-w}} \sum_{e|d} \frac{\mu(e)\chi(e)}{e^w}\\  \sum_{\substack{c_1\equiv 0(p)\\ (c_1,n/d)=1}}\frac{\psi(c_1)}{c_1^{s_1+s_2-w}} \sum_{f|c_1} \frac{\mu(f)\chi(f)}{f^{s_2}} \bigg] dwds_1ds_2. \end{multline}

Write $c_1=p c_1',$ renaming it $c_1$ and inverting the $f$-sum with the $c_1$-sum gives, \begin{multline}\label{eq:ym2}
\frac{\sqrt{pq}\psi(p)\chi(-q)}{(pq)^2}(\frac{1}{2\pi i})^3\int_{(s_1)} \int_{(s_2)}  \int_{(w)}  \frac{\Gamma(w)\Gamma(s_2-w)}{\Gamma(s_2)} \frac{(pq)^{s_1+s_2}}{p^{s_1+s_2-w}q^{s_2}} \\ \bigg[ L(\chi,w) \sum_{\substack{n \in \Z \\ n\neq 0\\ n \ll (pq)^\theta L^2}} \frac{\widetilde{H_n}(s_1,s_2)}{n^w} \sum_{r(n)} \frac{1}{r^{s_2-w}} e(\frac{d^{\circ}mr+\frac{l_1 l_2m}{d^{\circ}k^2}\overline{r}}{n})\overline{\psi(n)} \sum_{d|n} \frac{\psi(d)}{d^{s_1+s_2-w}} \sum_{e|d} \frac{\mu(e)\chi(e)}{e^w}\\   \sum_{\substack{f=1\\ (f,n/d)=1}} \frac{\mu(f)\chi(f)\psi(f)}{f^{s_1+2s_2-w}} \sum_{\substack{c_1=1\\ (c_1,n/d)=1}}\frac{\psi(c_1)}{c_1^{s_1+s_2-w}} \bigg] dwds_1ds_2. \end{multline}

Invert  the $d$- and $n$-sum to get \begin{multline}\label{eq:ym3}
\frac{\sqrt{pq}\psi(p)\chi(-q)}{(pq)^2}(\frac{1}{2\pi i})^3\int_{(s_1)} \int_{(s_2)}  \int_{(w)}  \frac{\Gamma(w)\Gamma(s_2-w)}{\Gamma(s_2)} \frac{(pq)^{s_1+s_2}}{p^{s_1+s_2-w}q^{s_2}} \\ \bigg[L(\chi,w) \sum_{d=1} \frac{1}{d^{s_1+s_2}} \sum_{e|d} \frac{\mu(e)\chi(e)}{e^w} \sum_{\substack{n \in \Z \\ n\neq 0\\ dn \ll (pq)^\theta L^2}} \frac{\widetilde{H_{dn}}(s_1,s_2)}{n^w} \sum_{r(dn)} \frac{1}{r^{s_2-w}} e(\frac{d^{\circ}mr+\frac{l_1 l_2m}{d^{\circ}k^2}\overline{r}}{dn})\overline{\psi(n)} \\   \sum_{\substack{f=1\\ (f,n)=1}} \frac{\mu(f)\chi(f)\psi(f)}{f^{s_1+2s_2-w}} \sum_{\substack{c_1=1\\ (c_1,n)=1}}\frac{\psi(c_1)}{c_1^{s_1+s_2-w}}  \bigg] dwds_1ds_2. \end{multline}

Similarly, invert the $e$- and $d$-sums to get

\begin{multline}\label{eq:ym4}\frac{\sqrt{pq}\psi(p)\chi(-q)}{(pq)^2}(\frac{1}{2\pi i})^3\int_{(s_1)} \int_{(s_2)}    \int_{(w)}   \frac{\Gamma(w)\Gamma(s_2-w)}{\Gamma(s_2)} \frac{(pq)^{s_1+s_2}}{p^{s_1+s_2-w}q^{s_2}}\\ \bigg[L(\chi,w) \sum_{e=1} \frac{\mu(e)\chi(e)}{e^{w+s_1+s_2}} \sum_{d=1} \frac{1}{d^{s_1+s_2}}  \sum_{\substack{n \in \Z \\ n\neq 0\\ edn \ll (pq)^\theta L^2}} \frac{\widetilde{H_{edn}}(s_1,s_2)}{n^w} \sum_{r(edn)} \frac{1}{r^{s_2-w}} e(\frac{d^{\circ}mr+\frac{l_1 l_2m}{d^{\circ}k^2}\overline{r}}{edn})\overline{\psi(n)} \\   \sum_{\substack{f=1\\ (f,n)=1}} \frac{\mu(f)\chi(f)\psi(f)}{f^{s_1+2s_2-w}} \sum_{\substack{c_1=1\\ (c_1,n)=1}}\frac{\psi(c_1)}{c_1^{s_1+s_2-w}}  \bigg] dwds_1ds_2. \end{multline}

Finally, we remove the relatively prime conditions on the $f$- and $c_1$-sums 

\begin{multline}\label{eq:ym5}\frac{\sqrt{pq}\psi(p)\chi(-q)}{(pq)^2}(\frac{1}{2\pi i})^3\int_{(s_1)} \int_{(s_2)}    \int_{(w)}   \frac{\Gamma(w)\Gamma(s_2-w)}{\Gamma(s_2)} \frac{(pq)^{s_1+s_2}}{p^{s_1+s_2-w}q^{s_2}}\\ \bigg[L(\chi,w) \sum_{e=1} \frac{\mu(e)\chi(e)}{e^{w+s_1+s_2}} \sum_{d=1} \frac{1}{d^{s_1+s_2}} \sum_{a=1} \frac{\mu(a)}{a^{s_1+s_2}} \sum_{b=1}  \frac{\mu^2(b)\chi(b)}{b^{s_1+2s_2}} \sum_{\substack{n \in \Z \\ n\neq 0\\ baedn \ll (pq)^\theta L^2}} \frac{\widetilde{H_{baedn}}(s_1,s_2)}{n^w}\\  \sum_{r(baedn)} \frac{1}{r^{s_2-w}} e(\frac{d^{\circ}mr+\frac{l_1 l_2m}{d^{\circ}k^2}\overline{r}}{baedn})\overline{\psi(n)}    \sum_{f=1} \frac{\mu(f)\chi(f)\psi(f)}{f^{s_1+2s_2-w}} \sum_{c_1=1}\frac{\psi(c_1)}{c_1^{s_1+s_2-w}}  \bigg] dwds_1ds_2. \end{multline}

This can be simplified to \begin{multline}\label{eq:ym6} \frac{\sqrt{pq}\psi(p)\chi(-q)}{(pq)^2}(\frac{1}{2\pi i})^3\int_{(s_1)} \int_{(s_2)}    \int_{(w)}   \frac{\Gamma(w)\Gamma(s_2-w)}{\Gamma(s_2)} \frac{(pq)^{s_1+s_2}}{p^{s_1+s_2-w}q^{s_2}}\\ \frac{L(\chi,w)L(\psi,s_1+s_2-w)}{L(\chi\psi,s_1+2s_2-w)} \bigg[\sum_{e=1} \frac{\mu(e)\chi(e)}{e^{w+s_1+s_2}} \sum_{d=1} \frac{1}{d^{s_1+s_2}} \sum_{a=1} \frac{\mu(a)}{a^{s_1+s_2}} \sum_{b=1}  \frac{\mu^2(b)\chi(b)}{b^{s_1+2s_2}} \\ \sum_{\substack{n \in \Z \\ n\neq 0\\ baedn \ll (pq)^\theta L^2}} \frac{\widetilde{H_{baedn}}(s_1,s_2)}{n^w} \sum_{r(baedn)} \frac{1}{r^{s_2-w}} e(\frac{d^{\circ}mr+\frac{l_1 l_2m}{d^{\circ}k^2}\overline{r}}{baedn})\overline{\psi(n)} \bigg] dwds_1ds_2. \end{multline}

\end{proof}

Remember $\Re(s_1)=2,$ $\Re(s_2)=1/2+\delta,$ and $\Re(w)=1/2$ with $\delta>0$ small. We pick up a main term by shifting $s_2$ to the line $\Re(s_2)=1/2-\delta$ with a residue term   \begin{multline}\label{eq:ymr}\frac{\sqrt{pq}\psi(p)\chi(-q)}{(pq)^2}(\frac{1}{2\pi i})^2\int_{(s_1)=2} \int_{(s_2)=1/2}       \frac{(pq)^{s_1+s_2}}{p^{s_1}q^{s_2}}\\ \frac{L(\chi,s_2)L(\psi,s_1)}{L(\chi\psi,s_1+s_2)} \bigg[\sum_{e=1} \frac{\mu(e)\chi(e)}{e^{s_1+2s_2}} \sum_{d=1} \frac{1}{d^{s_1+s_2}} \sum_{a=1} \frac{\mu(a)}{a^{s_1+s_2}} \sum_{b=1}  \frac{\mu^2(b)\chi(b)}{b^{s_1+2s_2}} \\ \sum_{\substack{n \in \Z \\ n\neq 0\\ baedn \ll (pq)^\theta L^2}} \frac{\widetilde{H_{baedn}}(s_1,s_2)}{n^{s_2}} \sum_{r(baedn)} e(\frac{d^{\circ}mr+\frac{l_1 l_2m}{d^{\circ}k^2}\overline{r}}{baedn})\overline{\psi(n)} \bigg] ds_1ds_2, \end{multline}

plus a remainder term  \begin{multline}\label{eq:yrem} \frac{\sqrt{pq}\psi(p)\chi(-q)}{(pq)^2}(\frac{1}{2\pi i})^3\int_{(s_1)=2} \int_{(s_2)=1/2-\delta}    \int_{(w)=1/2}  \frac{\Gamma(w)\Gamma(s_2-w)}{\Gamma(s_2)} \frac{(pq)^{s_1+s_2}}{p^{s_1+s_2-w}q^{s_2}}\\ \frac{L(\chi,w)L(\psi,s_1+s_2-w)}{L(\chi\psi,s_1+2s_2-w)} \bigg[\sum_{e=1} \frac{\mu(e)\chi(e)}{e^{w+s_1+s_2}} \sum_{d=1} \frac{1}{d^{s_1+s_2}} \sum_{a=1} \frac{\mu(a)}{a^{s_1+s_2}} \sum_{b=1}  \frac{\mu^2(b)\chi(b)}{b^{s_1+2s_2}} \\ \sum_{\substack{n \in \Z \\ n\neq 0\\ baedn \ll (pq)^\theta L^2}} \frac{\widetilde{H_{baedn}}(s_1,s_2)}{n^w} \sum_{r(baedn)} \frac{1}{r^{s_2-w}} e(\frac{d^{\circ}mr+\frac{l_1 l_2m}{d^{\circ}k^2}\overline{r}}{baedn})\overline{\psi(n)} \bigg] dwds_1ds_2. \end{multline}

\begin{remark}

The ``remainder" term is actually bigger then the ``main" residual term, but only by a fixed but arbitrarily small amount. We expect that this ``remainder" term via some kind of symmetry (perhaps a functional equation) is smaller than the ``main" term. We hope that something similar to the analysis in getting the main term of \cite{So} will work in this case. Regardless, we focus on the ``main" term first. 

\end{remark}

\subsection{Including the $m$-sum}\label{myoy}
Now we want to include the $m$-sum from \eqref{eq:kk} into \eqref{eq:ymr} and deal with the remainder later in Section \ref{remb}. The term to consider is \begin{multline}\label{eq:ywm} \frac{\sqrt{pq}  \psi(p)\chi(-q)}{(pq)^2} \sum_{m\geq 1} \frac{1}{m^{1/2}}F'_{M'}(\frac{d^{\circ}m}{pq})
(\frac{1}{2\pi i})^3\int_{(s_1)=2} \int_{(s_2)=1/2}      \frac{(pq)^{s_1+s_2}}{p^{s_1}q^{s_2}}\\ \frac{L(\chi,s_2)L(\psi,s_1)}{L(\chi\psi,s_1+s_2)} \bigg[\sum_{e=1} \frac{\mu(e)\chi(e)}{e^{s_1+2s_2}} \sum_{d=1} \frac{1}{d^{s_1+s_2}} \sum_{a=1} \frac{\mu(a)}{a^{s_1+s_2}} \sum_{b=1}  \frac{\mu^2(b)\chi(b)}{b^{s_1+2s_2}} \\ \sum_{\substack{n \in \Z \\ n\neq 0\\ baedn \ll (pq)^\theta L^2}} \frac{\widetilde{H_{baedn}}(s_1,s_2)}{n^{s_2}} \sum_{r(baedn)} e(\frac{d^{\circ}mr+\frac{l_1 l_2m}{d^{\circ}k^2}\overline{r}}{baedn})\overline{\psi(n)} \bigg] ds_1ds_2. \end{multline}

 Recall $\widetilde{H_{baedn}}(s_1,s_2)$ depends on $m.$ Specifically, \begin{equation}\label{eq:hst} \widetilde{H_\mu}(s_1,s_2)=\int_0^\infty \int_0^\infty e(\frac{x\dd m}{\mu y}+\frac{y\frac{l_1 l_2m}{\dd k^2}}{\mu x})\int_{-\infty}^{\infty}
e(\frac{t\mu }{xy})F_{M}(t)
 V(\frac{4\pi
\sqrt{t\dd m}}{\sqrt{pq}x})  W(\frac{4\pi \sqrt{t\frac{l_1 l_2m}{\dd k^2}}}{\sqrt{pq}y})x^{s_1-1}y^{s_2-1}\frac{dt}{t^{1/2}}dxdy.\end{equation}

Make a change of variables $x \to \frac{\sqrt{m\dd}x}{\sqrt{pq}}$ and $y \to \frac{\sqrt{m\frac{l_1l_2}{\dd k^2}}y}{\sqrt{pq}}$ to get $\widetilde{H_\mu}(s_1,s_2)$ equal to \begin{multline}\label{eq:hst0} (\frac{l_1l_2m}{k^2})^{\frac{s_1+s_2}{2}} \frac{1}{(pq)^{\frac{s_1+s_2}{2}} } \int_0^\infty \int_0^\infty e(\frac{x{\dd}^2 k m}
 {\sqrt{l_1l_2}hy}+\frac{y(l_1 l_2)^{3/2}m}{{\dd}^2 k^3hx}) \\ \int_{-\infty}^{\infty}
e(\frac{t\mu k }{\sqrt{l_1l_2}mxy})F_{M}(t)
 V(\frac{4\pi
\sqrt{t }}{x})  W(\frac{4\pi \sqrt{t}}{y})x^{s_1-1}y^{s_2-1}\frac{dt}{t^{1/2}}dxdy.\end{multline}

To simplify notation and to focus on the $m$-variable define \begin{multline} A_{s_1,s_2}(z):=z^{\frac{s_1+s_2}{2}}\int_0^\infty \int_0^\infty e(\frac{x{\dd}^2 k z}
 {\sqrt{l_1l_2}baedny}+\frac{y(l_1 l_2)^{3/2}z}{{\dd}^2 k^3baednx})\times \\ \int_{-\infty}^{\infty}
e(\frac{tbaednk}{\sqrt{l_1l_2}zxy})F_{M}(t)
 V(\frac{4\pi
\sqrt{t }}{x})  W(\frac{4\pi \sqrt{t}}{y})x^{s_1-1}y^{s_2-1}\frac{dt}{t^{1/2}}dxdy.  \end{multline}

 We interchange the integrals and the $e,d,a,b$-sums with the $m$-sum. This is ok as the $m$-sum is finite. So for the bracketed term of \eqref{eq:ywm} along with $A_{s_1,s_2}(m)$ (which is just $\widetilde{H_n}(s_1,s_2)$ relabelled)  and the $m$-sum we have 
 \begin{multline}\label{eq:msum8}
\sum_{e=1} \frac{\mu(e)\chi(e)}{e^{s_1+2s_2}} \sum_{d=1} \frac{1}{d^{s_1+s_2}} \sum_{a=1} \frac{\mu(a)}{a^{s_1+s_2}} \sum_{b=1}  \frac{\mu^2(b)\chi(b)}{b^{s_1+2s_2}} \\ \sum_{\substack{n \in \Z \\ n\neq 0\\ baedn \ll (pq)^\theta L^2}} \frac{1}{n^{s_2}}  \sum_{m\geq 1} \frac{1}{m^{1/2}}F'_{M'}(\frac{\dd m}{pq})  \sum_{r(baedn)} e(\frac{\dd mr+\frac{l_1 l_2m}{\dd k^2}\overline{r}}{baedn})A_{s_1,s_2}(m).
\end{multline}

Performing Poisson summation in $m$ modulo $baedn$ gives \begin{equation}
\label{eq:msum18}
\frac{1}{baedn}\sum_m \nu_{\dd}(baedn,m,\frac{l_1 l_2}{\dd k^2})\int_{-\infty}^\infty F'_{M'}(\frac{\dd v}{pq})A_{s_1,s_2}(v)e(\frac{-mv}{baedn})\frac{d v}{v^{1/2}},
\end{equation}
where $\nu_b(n,m,a):=\#\{x: ax^2-mx+b\equiv 0(n)\}.$

Following \cite{R}, $\nu_b(n,m,a)$ is multiplicative in $n,$ and we can reduce the study of it to prime power modulus $p^k.$ Suppose first we take $n=p,$ $p$ odd and $(a,p)=1,$ then completing the square $$\nu_b(p,m,a)=\#\{x: x^2\equiv  (m\overline{2a})^2-b\overline{a}(p)\}.$$ Again as in \cite{R} using Hensel's lemma gives $\nu_b(p^k,m,a)=\nu_b(p,m,q)$ for $k>1.$ Then its clear \begin{equation}\label{eq:rru} \nu_b(p,m,a)=1+\left(\frac{m^2-4ba}{p}\right)=\left\{ \begin{array}{ll}
      2 & \text{if } m^2-4ba\equiv z^2(p)  \\
        0 & \text{otherwise} .\end{array} \right. \end{equation}
Using multiplicativity, we have \begin{equation}\label{eq:cmnu} \nu_{\dd}(\alpha\beta,m,\frac{l_1 l_2}{\dd k^2})=\nu_{\dd}(\mu^2(\alpha\beta),m,\frac{l_1 l_2}{\dd k^2})=\nu_{\dd}(\mu^2(\alpha),m,\frac{l_1 l_2}{\dd k^2})\nu_{\dd}(\mu^2(\beta),m,\frac{l_1 l_2}{\dd k^2}),\end{equation} where $\mu$ is the Mobius function. The last equality implies we can extend this result to square free modulus for $ \nu_{\dd}(baedn,m,\frac{l_1 l_2}{\dd k^2}).$ 

\subsubsection{Isolating the $b,a,e,d,n$-sums}

In order to isolate these sums, we open up $A_{s_1,s_2}(v)$ and focus on the terms containing $b,a,e,d,$ and $n.$ Let \begin{multline}\label{eq:ttt} T_{m,v}(c):=e(\frac{-mv}{c}) \int_0^\infty \int_0^\infty e(\frac{x{\dd}^2 k v}
 {\sqrt{l_1l_2}cy}+\frac{y(l_1 l_2)^{3/2}v}{{\dd}^2 k^3cx}) \\ \int_{-\infty}^{\infty}
e(\frac{tc k }{\sqrt{l_1l_2}vxy})F_{M}(t)
 V(\frac{4\pi
\sqrt{t }}{x})  W(\frac{4\pi \sqrt{t}}{y})x^{s_1-1}y^{s_2-1}\frac{dt}{t^{1/2}}dxdy.\end{multline}

Putting \eqref{eq:ttt} back into \eqref{eq:msum8}, using the multiplicativity of \eqref{eq:cmnu} above, and interchanging the $m$-sum (after Poisson summation) with the other sums gives \begin{multline}\label{eq:aftp}\sum_m \int_{-\infty}^\infty F'_{M'}(\frac{\dd v}{pq})\bigg[ \sum_{e=1} \frac{\mu(e)\chi(e)\nu_{\dd}(e,m,\frac{l_1 l_2}{\dd k^2})}{e^{s_1+2s_2+1}} \sum_{d=1} \frac{\nu_{\dd}(d,m,\frac{l_1 l_2}{\dd k^2})}{d^{s_1+s_2+1}} \sum_{a=1} \frac{\mu(a)\nu_{\dd}(a,m,\frac{l_1 l_2}{\dd k^2})}{a^{s_1+s_2+1}} \\ \sum_{b=1}  \frac{\mu^2(b)\chi(b)\nu_{\dd}(b,m,\frac{l_1 l_2}{\dd k^2})}{b^{s_1+2s_2+1}} \sum_{\substack{n \in \Z \\ n\neq 0\\ baedn \ll (pq)^\theta L^2}} \frac{\overline{\psi(n)} \nu_{\dd}(n,m,\frac{l_1 l_2}{\dd k^2})T_{m,v}(baedn)}{n^{s_2+1}} \bigg] v^{\frac{s_1+s_2-1}{2}} d v.\end{multline}

 To completely isolate the $b,a,e,d,n$-sums in the archimedean functions requires applying Mellin inversion to $T_{m,v}(c).$  It is easy to check by integration by parts in the $t$-variable that $$T_{m,v}(c) \ll \frac{1}{(1+|\Im(c)|)^N}$$ for any $N>0.$  The variable of the Mellin inversion is $w$ and we integrate on the line $\Re(w)=\epsilon,$ for $\epsilon >0$ small. Using the above estimate on $T_{m,v}(c)$ we move the $e,d,a,b,n$-sums inside of the $w$-integral to get
 
 \begin{multline}\label{eq:aftp0}\sum_m \int_{-\infty}^\infty F'_{M'}(\frac{\dd v}{pq})\bigg(\frac{1}{2\pi i} \int_{(w)=\epsilon}\widetilde{T_{m,v}}(w) \bigg[ \sum_{e=1} \frac{\mu(e)\chi(e)\nu_{\dd}(e,m,\frac{l_1 l_2}{\dd k^2})}{e^{s_1+2s_2+1+w}} \sum_{d=1} \frac{\nu_{\dd}(d,m,\frac{l_1 l_2}{\dd k^2})}{d^{s_1+s_2+1+w}}\\  \sum_{a=1} \frac{\mu(a)\nu_{\dd}(a,m,\frac{l_1 l_2}{\dd k^2})}{a^{s_1+s_2+1+w}}  \sum_{b=1}  \frac{\mu^2(b)\chi(b)\nu_{\dd}(b,m,\frac{l_1 l_2}{\dd k^2})}{b^{s_1+2s_2+1+w}} \sum_{\substack{n \in \Z \\ n\neq 0\\ baedn \ll (pq)^\theta L^2}} \frac{\overline{\psi(n)} \nu_{\dd}(n,m,\frac{l_1 l_2}{\dd k^2})}{n^{s_2+1+w}} \bigg] dw \bigg) v^{\frac{s_1+s_2-1}{2}} d v.\end{multline}



Let $\Delta:=m^2-4\frac{l_1l_2}{k^2},$ and $S:=\{\mathfrak{p} \text{ prime} 
: \mathfrak{p} \mid 2q\Delta \}.$ As the usual notation for primes is $``p",$ we use {\it mathfrak} for primes of our Euler product below. Hopefully this does not cause too much confusion. Either way, this possible confusion of notation ends at in this section. Following \cite{R}, for a fixed prime $\mathfrak{p} \notin S,$ we have \begin{equation}\label{eq:lll} 1+\sum_{k=1} ^\infty \frac{\nu_{\dd}(\mathfrak{p},m,\frac{l_1 l_2}{\dd k^2})\overline{\psi}(\mathfrak{p}^k)}{\mathfrak{p}^{kr}}=1+ \frac{(1+(\frac{\Delta}{\mathfrak{p}}))\frac{\overline{\psi}(\mathfrak{p})}{\mathfrak{p}^{r}}}{(1-\frac{\overline{\psi}(\mathfrak{p})}{\mathfrak{p}^{r}})}= \frac{(1+\frac{\overline{\psi}(\mathfrak{p})(\frac{\Delta}{\mathfrak{p}}))}{\mathfrak{p}^r})}{(1-\frac{\overline{\psi}(\mathfrak{p})}{\mathfrak{p}^r})}.\end{equation}

Now look at the individual sums in \eqref{eq:aftp0}. Note with a negligible error we can remove the restriction $$ baedn \ll (pq)^\theta L^2.$$ Therefore, the $e$-sum following what we did for \eqref{eq:lll} can be written as an Euler product  equaling  \begin{equation}\label{eq:es} \prod_{\mathfrak{p}\notin S} \frac{(1-\frac{\chi(\mathfrak{p})}{\mathfrak{p}^{s_1+2s_2+w}})}{(1+\frac{\chi(\mathfrak{p})(\frac{\Delta}{\mathfrak{p}}))}{\mathfrak{p}^{s_1+2s_2+w}})}  \prod_{\mathfrak{p} \in S}\frac{1}{L_{\mathfrak{p},\Delta}(\chi, s_1+2s_2+w)}.\end{equation} Here $L_{\mathfrak{p},\Delta}(\chi, s)$ we do not make specific, but note it is bounded in the range $\Re(s)\geq1$ by $ (\Delta pq)^{\epsilon}$ using \eqref{eq:con}.

Notice the $a$-sum and $d$-sum cancel by standard Euler product arguments. 

The $n$-sum equals \begin{equation}\label{eq:ns} \prod_{\mathfrak{p}\notin S} \frac{(1+\frac{\overline{\psi}(\mathfrak{p})(\frac{\Delta}{\mathfrak{p}}))}{\mathfrak{p}^{s_2+1+w}})}{(1-\frac{\overline{\psi}(\mathfrak{p})}{\mathfrak{p}^{s_2+1+w}})}  \prod_{\mathfrak{p} \in S} L_{\mathfrak{p},\Delta}(\overline{\psi}, s_2+1+w).\end{equation}

The $b$-sum is more difficult and its Euler product equals \begin{equation}\label{eq:bs}\prod_{\mathfrak{p} \notin S} \frac{(1-\frac{\chi^2(\mathfrak{p})(1+(\frac{\Delta}{\mathfrak{p}}))^2}{\mathfrak{p}^{2s_1+4s_2+1+2w}})}{(1-\frac{\chi(\mathfrak{p})(1+(\frac{\Delta}{\mathfrak{p}}))}{\mathfrak{p}^{s_1+2s_2+1+w}})}\prod_{\mathfrak{p} \in S} B_{\mathfrak{p},\Delta}(s_1+2s_2+1+w)=:B(s_1+2s_2+1+w).\end{equation} Here we do not explicate $B_{\mathfrak{p},\Delta}(s)$ but analogous analysis to that of the $e$-sum we have $$\prod_{\mathfrak{p} \in S} B_{\mathfrak{p},\Delta}(s) \ll  (\Delta pq)^{\epsilon}$$ for $\Re(s)\geq 1.$

The analytic behavior of the $e$-sum \eqref{eq:es} is the same as  $$\frac{L(s_1+2s_2+1+w,(\frac{\Delta}{\cdot})\chi)}{L(s_1+2s_2+1+w,\chi)}$$ up to the factors $p \in S.$ Likewise, the $n$-sum \eqref{eq:ns} is the same as $$\frac{L(s_2+1+w,\overline{\psi})}{L(s_2+1+w,(\frac{\Delta}{\cdot})\overline{\psi}))}$$ and the $b$-sum is absolutely convergent for $\Re(s_1+2s_2+w) >1.$ We will not address further the analytic structure of the $b$-sum. Label the $b$-sum $B(s).$

We consolidate what we have from Proposition \ref{bigp} and the Poisson summation on the $m$-sum getting  \begin{multline}\label{eq:res} \frac{\sqrt{pq}  \psi(p)\chi(-q)}{(pq)^2}
\sum_m  (\frac{1}{2\pi i})^2 \int_{(s_1)=2} \int_{(s_2)=1/2}   (\frac{l_1l_2}{k^2})^{\frac{s_1+s_2}{2}}\int_{(w)=\epsilon}   \int_{-\infty}^\infty F'_{M'}(\frac{\dd v}{pq}) \widetilde{T_{m,v}}(w) \frac{1}{(pq)^{\frac{s_1+s_2}{2}}} \frac{(pq)^{s_1+s_2}}{p^{s_1}q^{s_2}}\\ \frac{L(s_2,\chi)L(s_1,\psi)}{L(s_1+s_2,\chi\psi)} \frac{L(s_1+2s_2+1+w,(\frac{\Delta}{\cdot})\chi)}{L(s_1+2s_2+1+w,\chi)} \frac{L(s_2+1+w,\overline{\psi})}{L(s_2+1+w,(\frac{\Delta}{\cdot})\overline{\psi}))} B(s_1+2s_2+1+w) v^{\frac{s_1+s_2-1}{2}} d v dw ds_1ds_2\\ +O((pq)^{-A}). \end{multline}

 The finite product of $L_{\mathfrak{p},\Delta}(s)$ for $\mathfrak{p} \in S$ in the above cases do not affect convergence issues. In other words, we can shift the contour of $s_1,s_2$ in \eqref{eq:ywm} by looking at these quotients of L-functions above. 

We aim to bound \eqref{eq:res} now.

\subsection{Bounding the $v$-integral} \label{vil}

The $v$-integral is \begin{multline}\label{eq:vint}
 \int_{-\infty}^\infty F'_{M'}(\frac{\dd v}{pq}) \widetilde{T_{m,v}}(w)v^{\frac{s_1+s_2-1}{2}} d v=\int_0^\infty \int_{-\infty}^\infty F'_{M'}(\frac{\dd v}{pq}) e(\frac{-mv}{h}) \int_0^\infty \int_0^\infty e(\frac{x{\dd}^2 k v}
 {\sqrt{l_1l_2}hy}+\frac{y(l_1 l_2)^{3/2}v}{{\dd}^2 k^3hx}) \\ \int_{-\infty}^{\infty}
e(\frac{th k }{\sqrt{l_1l_2}vxy})F_{M}(t)
 V(\frac{4\pi
\sqrt{t }}{x})  W(\frac{4\pi \sqrt{t}}{y})\frac{dt}{t^{1/2}} x^{s_1-1}y^{s_2-1}dxdy h^{w-1}dh v^{\frac{s_1+s_2-1}{2}} d v.
\end{multline}

Change variables $h \to hv,$ $x \to \sqrt{t}x$ and $y \to \sqrt{t}y$ we can write \eqref{eq:vint} as \begin{multline}\label{eq:vint0}
 \int_{-\infty}^\infty F'_{M'}(\frac{\dd v}{pq}) \widetilde{T_{m,v}}(w)v^{\frac{s_1+s_2}{2}} d v=\bigg[\int_0^\infty F_{M}(t)t^{\frac{s_1+s_2-1}{2}}dt\bigg] \left[\int_{-\infty}^\infty F'_{M'}(\frac{\dd v}{pq})v^{\frac{s_1+s_2+2w-1}{2}} d v\right] \times \\ \int_0^\infty \int_0^\infty  V(\frac{4\pi}{x})  W(\frac{4\pi}{y}) \bigg( \int_0^{\infty}  e(\frac{-m}{h})  e(\frac{x{\dd}^2 k }
 {\sqrt{l_1l_2}hy}+\frac{y(l_1 l_2)^{3/2}}{{\dd}^2 k^3hx}) 
e(\frac{h k }{\sqrt{l_1l_2}xy})  h^{w-1}dh\bigg) x^{s_1-1}y^{s_2-1}dxdy .
\end{multline}

Note that the $v$-integral is now bounded by $(pq)^{\epsilon}$ for any $\epsilon>0$ with $s_1=s_2=1/2.$ The $x,y,$ and $t$ integrals are bounded and so we focus on the $h$-integral. 

We completely isolate the test functions $F_M,F'_{M'}$ by changing $v \to \frac{pqv}{\dd}$ getting \eqref{eq:vint0} equaling  \begin{multline}\label{eq:vint00}
(\frac{pq}{\dd})^{\frac{s_1+s_2+2w+1}{2}} \bigg[\int_0^\infty F_{M}(t)t^{\frac{s_1+s_2-1}{2}}dt\bigg] \left[\int_{-\infty}^\infty F'_{M'}(v)v^{\frac{s_1+s_2+2w-1}{2}} d v\right] \int_0^\infty \int_0^\infty  V(\frac{4\pi}{x})  W(\frac{4\pi}{y}) \\ \bigg( \int_0^{\infty}  e(\frac{-m}{h})  e(\frac{x{\dd}^2 k }
 {\sqrt{l_1l_2}hy}+\frac{y(l_1 l_2)^{3/2}}{{\dd}^2 k^3hx}) 
e(\frac{h k }{\sqrt{l_1l_2}xy})  h^{w-1}dh\bigg) x^{s_1-1}y^{s_2-1}dxdy .
\end{multline}
\begin{remark}
So note the choice of $F_M,F'_{M'}$ from the approximate functional equation in \eqref{eq:approx} is independent of the analysis needed in studying the subconvexity problem for the Rankin-Selberg L-function. Therefore if we want, we can bring the $M$-sum back in and remove the partition of unity. 
\end{remark}

We will eventually need an estimate for the $m$-sum to converge. Changing variables from $h \to h^{-1}$ and integrating by parts $j$-times easily gives the bound for the $h$-integral \begin{equation}\label{eq:mconv}
\int_0^{\infty}  e(\frac{-m}{h})  e(\frac{x{\dd}^2 k }
 {\sqrt{l_1l_2}hy}+\frac{y(l_1 l_2)^{3/2}}{{\dd}^2 k^3hx}) 
e(\frac{h k }{\sqrt{l_1l_2}xy})  h^{w-1}dh \ll \left(\frac{(l_1l_2)^{3/2}}{{\dd}^2 k m}\right)^j.
\end{equation}

Define $$A:=A(x,y,\dd,k,l_1,l_2):=\left(\frac{x{\dd}^2 k }
 {\sqrt{l_1l_2}y}+\frac{y(l_1 l_2)^{3/2}}{{\dd}^2 k^3x}-m\right)$$ and $$B:=B(x,y,k,l_1,l_2):=\frac{k}{\sqrt{l_1l_2}xy}.$$

We need to bound the integral, \begin{equation}\label{eq:bess}
\int_0^\infty e(\frac{A}{h}+Bh)h^{w-1}dh \end{equation} in terms of $pq.$ We write the $h$-integral in terms of Bessel functions by recalling the following identities from \cite{IK}, \begin{equation}
\int_0^\infty \cos(\frac{x}{2}(y+\frac{1}{y}))y^{s-1}dy=-\pi J_s(x)\sin(\pi s/2)-\pi Y_s(x)\cos(\pi s/2)=-\pi\frac{J_s(x)-J_{-s}(x)}{2\sin(\pi s/2)},
\end{equation}
and
\begin{equation}
\int_0^\infty \sin(\frac{x}{2}(y+\frac{1}{y}))y^{s-1}dy=-\pi J_s(x)\cos(\pi s/2)-\pi Y_s(x)\sin(\pi s/2)=-\pi\frac{J_s(x)+J_{-s}(x)}{2\cos(\pi s/2)}.
\end{equation}

So \begin{multline}\label{eq:bess1}
\int_0^\infty e(\frac{A}{h}+Bh)h^{w-1}dh=-\pi (\frac{A}{B})^{w/2}\bigg[\frac{J_w(4\pi\sqrt{\frac{A}{B}})-J_{-w}(4\pi\sqrt{\frac{A}{B}})}{2\sin(\pi w/2)} +\frac{J_w(4\pi\sqrt{\frac{A}{B}})+J_{-w}(4\pi\sqrt{\frac{A}{B}})}{2\cos(\pi w/2)}\bigg].
\end{multline}

Define \begin{equation}\label{eq:ccc} C(s_1,s_2,w):=\frac{L(s_2,\chi)L(s_1,\psi)}{L(s_1+s_2,\chi\psi)} \frac{L(s_1+2s_2+1+w,(\frac{\Delta}{\cdot})\chi)}{L(s_1+2s_2+1+w,\chi)} \frac{L(s_2+1+w,\overline{\psi})}{L(s_2+1+w,(\frac{\Delta}{\cdot})\overline{\psi}))} B(s_1+2s_2+1+w)\end{equation} as well as \begin{multline}\label{eq:ddd}D_{\frac{A}{B}}(s_1,s_2,w)= -\pi  \int_0^\infty \int_0^\infty V(\frac{4\pi}{x})  W(\frac{4\pi }{y}) \\  (\frac{A}{B})^{w/2}\bigg[\frac{J_w(4\pi\sqrt{\frac{A}{B}})-J_{-w}(4\pi\sqrt{\frac{A}{B}})}{2\sin(\pi w/2)} +\frac{J_w(4\pi\sqrt{\frac{A}{B}})+J_{-w}(4\pi\sqrt{\frac{A}{B}})}{2\cos(\pi w/2)}\bigg] x^{s_1-1}y^{s_2-1}dxdy.\end{multline}

 Then incorporating \eqref{eq:ccc} and \eqref{eq:ddd} into \eqref{eq:res} gives 

 \begin{multline}\label{eq:res0} \frac{\psi(p)\chi(-q)}{pq}
\sum_m  (\frac{1}{2\pi i})^2 \int_{(s_1)=2} \int_{(s_2)=1/2}  \frac{(pq)^{s_1+s_2}}{p^{s_1}q^{s_2}}   (\frac{l_1l_2}{k^2})^{\frac{s_1+s_2}{2}}\times \\ \int_{(w)=\epsilon} (\frac{pq}{\dd})^w \bigg[\int_0^\infty F_{M}(t)t^{\frac{s_1+s_2-1}{2}}dt\bigg]  \left[\int_{-\infty}^\infty F'_{M'}(v)v^{\frac{s_1+s_2+2w-1}{2}} d v\right]\times \\ D_{\frac{A}{B}}(s_1,s_2,w)C(s_1,s_2,w) dw ds_1ds_2+ O((pq)^{-A}). \end{multline} We now make a contour shift on the line $\Re(s_1)=2$ to $\Re(s_1)=1/2.$ We encounter no poles in this shift.

Now write $s_1=1/2+it_1,s_2=1/2+it_2,$ and $w=iy.$ We now bound the $w$-integral.  First $$C(1/2+it_1,1/2+it_2,iy) \ll (|t_1|p)^{\theta_1} (|t_2|q)^{\theta_2} (|t_1t_2y|pq)^{\epsilon},$$ where $\theta_1,\theta_2$ are parameters $\leq 1/4.$  Here we are using  \begin{equation}\label{eq:con}(\frac{1}{|t|q})^{\epsilon} \ll L(\sigma+it,\chi) \ll (|t|q)^{\epsilon},\end{equation} for $\chi$ a character of modulus $q, \sigma\geq 1$ and $t \in \R.$ We also are using the (sub)convexity bound $$L(1/2+it,\chi) \ll (|t|q)^{\theta+\epsilon}$$ with $\theta \leq 1/4$ \cite{IK}.

Second, we bound \begin{equation} D_{\frac{A}{B}}(1/2+it_1,1/2+it_2,iy) .\end{equation} 

\begin{prop}\label{dyo}

\begin{equation}  D_{\frac{A}{B}}(1/2+it_1,1/2+it_2,i\gamma)  \ll (t_1t_2\gamma \frac{A}{B})^{\epsilon} \end{equation}

\end{prop}

\begin{proof}

As the $x,y$ integrals are bounded, we first look at the case $w=i\gamma \neq 0,$ it is sufficient to look at \begin{equation}\label{eq:xbes}
(\frac{A}{B})^{i\gamma/2} \frac{J_{2i\gamma}(\frac{4\pi\sqrt{A}}{\sqrt{B}})}{\sin(\pi i\gamma/2)}  \ll \frac{J_{2i\gamma}(\frac{4\pi\sqrt{A}}{\sqrt{B}})}{\sinh(\pi \gamma/2)}.  \end{equation}  

Let $T(x,y)= \frac{4\pi\sqrt{A}}{\sqrt{B}}.$
Fix $w=i\gamma\neq 0,$ then we break the analysis into two cases: $|\gamma| \geq T(x,y)$ and $|\gamma| \leq T(x,y).$ 

First we study the Bessel asymptotic for $T(x,y)$ smaller than the index. We need the behavior for $J_{2i\gamma}(T)$ for $\gamma \geq c T(x,y)$ and $c$ any positive number. Let $z=\sqrt{T^2+\gamma^2},$ then from (\cite{E},pg.87) \begin{multline}\label{eq:bigarg}
J_{2i\gamma}(T)=(2\pi^{1/2})^{-1} z^{-1/2} e^{\frac{-\pi i}{4}} \exp(\pi \gamma) e(\frac{z}{\pi}-\frac{\gamma}{\pi}\log(\frac{z-\gamma}{T})) \times \\ \left\{ 1+ \frac{1}{2i\gamma}(\frac{\gamma}{8z}-\frac{5\gamma^3}{24z^3})+\frac{1}{(2i\gamma)^2}(\frac{9\gamma^2}{128z^2}-\frac{231\gamma^4}{576z^4}+...)\right\}
\end{multline}

Incorporating this asymptotic for the appropriate range of integration, the $x,y$-integral can be be bounded up to a constant by

$$\int_0^\infty \int_0^\infty V(\frac{4\pi}{x})  W(\frac{4\pi }{y})  T(x,y)^{w}\frac{1}{w^{1/2}\sqrt{1+(T(x,y)/w)^2}} x^{s_1-1}y^{s_2-1}dxdy+\text{remainder terms}.$$ The main term is trivially bounded by $$\frac{|\frac{4\pi\sqrt{A}}{\sqrt{B}}|^{\Re(w)}}{|w|^{1/2}},$$ while the remainder terms are bounded by $$\frac{|\frac{4\pi\sqrt{A}}{\sqrt{B}}|^{\Re(w)}}{|w|^{1/2+k}}$$ for $k \in \N.$



Now we study the asymptotic for the argument $T(x,y)$ bigger than the index. For $T\geq |\gamma|,$ we use from (\cite{Wat} pg. 205) \begin{equation}
J_{2i\gamma}(T)=\frac{1}{\sqrt{2\pi T}}(W_1(2i\gamma,T)e^{iT}+W_2(2i\gamma,T)e^{-iT}),
\end{equation} where $$\frac{\partial^{(j)}}{\partial x^j}W_i(2iy,x) \ll_j (1+|x|)^{-j} \cosh(\pi y),$$ for $i=1,2,j \geq 0.$

So the $x,y$- integral, in this case, is bounded by \begin{multline}\label{eq:fff0}\int_0^\infty \int_0^\infty V(\frac{4\pi}{x})  W(\frac{4\pi }{y})  T(x,y)^{i\gamma-1/2}\left[\frac{W_1(2i\gamma,T(x,y))e^{iT(x,y)}+W_2(2i\gamma,T(x,y))e^{-iT(x,y)}}{\sinh(\pi \gamma)}\right]\times \\   x^{s_1-1}y^{s_2-1}dxdy\ll \frac{1}{\sqrt{T(x,y)}}\ll \frac{1}{|
\gamma|^{1/4}}.\end{multline}

Without loss of generality this argument works for the other Bessel functions divided by the hyperbolic functions in \eqref{eq:ddd}.

Now for the case $w=0,$ $$\lim_{w \to 0}\left[ \frac{J_w(4\pi T)-J_{-w}(4\pi T)}{2\sin(\pi w/2)} +\frac{J_w(4\pi T)+J_{-w}(4\pi T)}{2\cos(\pi w/2)} \right]=\frac{Y_0(4\pi T)}{2}+J_0(4\pi T).$$  Here we use the continuity of the index of the Bessel function (\cite{GR} 8.487) and the definition  $$Y_w(x)=\frac{J_w(x)\cos(\pi w)-J_{-w}(x)}{\sin(\pi w)}.$$ Assume $|T(x,y)|$ is larger than an absolute positive constant $C.$ Then for the second term $J_0(4\pi T)$ we can use the same analysis as is used to get \eqref{eq:fff0}. For the first term we can use the asymptotic $$Y_0(z)=\sqrt{\frac{2}{\pi z}} \sin(z-\frac{\pi}{4})+O(|z|^{-1}),$$ \cite{GR}. Similar to \eqref{eq:fff0} we can bound the $x,y$-integral up to an absolute constant by $ \frac{1}{\sqrt{T(x,y)}}\ll \frac{1}{\sqrt{C}}.$ 

If $|T(x,y)|$ is  smaller than $C,$ then we use the first term of the power series expansion for $J_0(4\pi T)$ and the asymptotic formula \cite{GR} $$Y_0(z)=\frac{2}{\pi}\left[ \ln(z/2)+\gamma\right],$$ with $\gamma$ Euler's constant. The integral in this case can be bounded by $\ln(2 \pi T) \ll T^{\epsilon}.$

If $T(x,y)=0,$ the $x,y$ integrals converge and are independent of the parameters $p,q.$




\end{proof}

 
 
 \subsection{Bounding the $s_1,s_2,w$-integrals}


Recall we are trying to bound   \begin{multline}\label{eq:res8} \frac{\psi(p)\chi(-q)}{pq}
\sum_m  (\frac{1}{2\pi i})^2 \int_{(s_1)=2} \int_{(s_2)=1/2}  \frac{(pq)^{s_1+s_2}}{p^{s_1}q^{s_2}}   (\frac{l_1l_2}{k^2})^{\frac{s_1+s_2}{2}}\times \\ \int_{(w)=\epsilon} (\frac{pq}{\dd})^w \bigg[\int_0^\infty F_{M}(t)t^{\frac{s_1+s_2-1}{2}}dt\bigg] \left[\int_{-\infty}^\infty F'_{M'}(v)v^{\frac{s_1+s_2+2w-1}{2}} d v\right]  D_{\frac{A}{B}}(s_1,s_2,w)C(s_1,s_2,w) dw ds_1ds_2\\+O((pq)^{-A}). \end{multline} By the previous section we have a bound for the functions $D_{\frac{A}{B}}(s_1,s_2,w)$ and $C(s_1,s_2,w).$ We must now ensure the $s_1,s_2,$ and $w$-integrals converge. 

Using integration by parts we have $$\int_0^\infty F'_{M'}(t)t^{\frac{s_1+s_2+2w-1}{2}}dt \ll \big(\frac{|s_1+s_2|)}{1+|w|}\big)^M$$ for any $M>0.$ We also have $$\int_0^\infty F_{M}(t)t^{\frac{s_1+s_2-1}{2}}dt \ll \frac{1}{(1+|s_1+s_2|)^N},$$ for any $N>0.$ Clearly, if we choose $N\geq M+2\geq 2$ all three of the integrals converge. 

  Remember the $m$-sum converges by the estimate \eqref{eq:mconv}. If we include the $l_1,l_2$-sums into \eqref{eq:res0}, we have 
\begin{multline}\label{eq:resl}   \frac{\psi(p)\chi(-q)}{pq}
 \sum_{l_1,l_2 \leq L} x_{l_1} x_{l_2}\sum_{k|(l_1,l_2)}  \sum_{\dd|{\frac{l_1l_2}{k^2}}} \frac{1}{{\dd}^{1/2}}
\sum_m  (\frac{1}{2\pi i})^2 \int_{(s_1)=1/2} \int_{(s_2)=1/2}  \frac{(pq)^{s_1+s_2}}{p^{s_1}q^{s_2}}   (\frac{l_1l_2}{k^2})^{\frac{s_1+s_2}{2}}\times \\ \int_{(w)=\epsilon} (\frac{pq}{\dd})^w  \bigg[\int_0^\infty F_{M}(t)t^{\frac{s_1+s_2-1}{2}}dt\bigg] \left[\int_{-\infty}^\infty F'_{M'}(v)v^{\frac{s_1+s_2+2w-1}{2}} d v\right]\times \\ D(s_1,s_2,w)C(s_1,s_2,w) dw ds_1ds_2+O((pq)^{-A}) \\ \ll \frac{p^{\theta_1+\epsilon}q^{\theta_2+\epsilon}}{\sqrt{pq}} \sum_{l_1,l_2 \leq L} x_{l_1} x_{l_2}l_1^{1/2}l_2^{1/2}\sum_{k|(l_1,l_2)} \frac{1}{k} \sum_{\dd|{\frac{l_1l_2}{k^2}}}\frac{1}{{\dd}^{1/2}}\ll \frac{p^{\theta_1+\epsilon}q^{\theta_2+\epsilon}}{\sqrt{pq}} L^{1+\epsilon}||x||_2^2   \end{multline} with $\theta_1,\theta_2 \leq 1/4.$

\section{Remainder term from \eqref{eq:yrem}}\label{remb}

In turns out in this case the ``remainder" term is bigger than the residue term with respect to $pq$ (but only by a fixed but arbitrary amount), however the analysis is essentially the same. We hope that there is a hidden symmetry similar to \cite{So} so that this ``remainder" term  cancels via some kind of functional equation and leaves the beautiful main term \eqref{eq:res0}. 

We study
\begin{multline}\label{eq:yrem00} \frac{\sqrt{pq}\psi(p)\chi(-q)}{(pq)^2}(\frac{1}{2\pi i})^3\int_{(s_1)=2} \int_{(s_2)=1/2-\delta}    \int_{(\alpha)=1/2}  \frac{\Gamma(\alpha)\Gamma(s_2-\alpha)}{\Gamma(s_2)} \frac{(pq)^{s_1+s_2}}{p^{s_1+s_2-\alpha}q^{s_2}}\\ \frac{L(\chi,\alpha)L(\psi,s_1+s_2-\alpha)}{L(\chi\psi,s_1+2s_2-\alpha)} \bigg[\sum_{e=1} \frac{\mu(e)\chi(e)}{e^{\alpha+s_1+s_2}} \sum_{d=1} \frac{1}{d^{s_1+s_2}} \sum_{a=1} \frac{\mu(a)}{a^{s_1+s_2}} \sum_{b=1}  \frac{\mu^2(b)\chi(b)}{b^{s_1+2s_2}} \\ \sum_{\substack{n \in \Z \\ n\neq 0\\ baedn \ll (pq)^\theta L^2}} \frac{\overline{\psi(n)}\widetilde{H_{baedn}}(s_1,s_2)}{n^\alpha} \sum_{r(baedn)} \frac{1}{r^{s_2-\alpha}} e(\frac{d^{\circ}mr+\frac{l_1 l_2m}{d^{\circ}k^2}\overline{r}}{baedn}) \bigg] dwds_1ds_2. \end{multline}

As in the main term we will execute Poisson summation in the $m$-sum modulo $baedn;$ the analogous interchanges of sums and integrals is identical.

\begin{equation}\label{eq:msum18}
\frac{1}{baedn}\sum_m \sum_{\substack{r(baedn)\\ \{r: \frac{l_1l_2r^2}{\dd k^2}-mr+\dd\equiv 0(baedn)\}}} \frac{1}{r^{s_2-\alpha}} \int_{-\infty}^\infty F'_{M'}(\frac{\dd v}{pq})A_{s_1,s_2}(v)e(\frac{-mv}{baedn})\frac{d v}{v^{1/2}}.\end{equation}

Applying the same steps up to \eqref{eq:res} gives  \begin{multline}\label{eq:resr} \frac{\sqrt{pq}  \psi(p)\chi(-q)}{(pq)^2}
\sum_m  (\frac{1}{2\pi i})^3 \int_{(s_1)=2} \int_{(s_2)=1/2-\delta}  (\frac{l_1l_2}{k^2})^{\frac{s_1+s_2}{2}} \int_{(\alpha)=1/2}  \frac{\Gamma(\alpha)\Gamma(s_2-\alpha)}{\Gamma(s_2)} \frac{(pq)^{s_1} (pq)^{s_2}}{p^{s_1+s_2-\alpha}q^{s_2}} \\   \frac{L(\alpha,\chi)L(s_1+s_2-\alpha,\psi)}{L(s_1+2s_2-\alpha,\chi\psi)} \int_{(w)=\epsilon}   \int_{-\infty}^\infty F'_{M'}(\frac{\dd v}{pq}) \widetilde{T_{m,v}}(w)    \bigg[ \sum_{e=1} \frac{\mu(e)\chi(e)}{e^{\alpha+s_1+2s_2+1+w}} \sum_{d=1} \frac{1}{d^{s_1+s_2+1+w}} \sum_{a=1} \frac{\mu(a)}{a^{s_1+s_2+1+w}} \\ \sum_{b=1}  \frac{\mu^2(b)\chi(b)}{b^{s_1+2s_2+1+w}} \sum_{\substack{n \in \Z\\ n \neq0}} \frac{\overline{\psi(n)} }{n^{\alpha+1+w}} \big(\sum_{\substack{r(baedn)\\ \{r: \frac{l_1l_2r^2}{\dd k^2}-mr+\dd\equiv 0(baedn)\}}} \frac{1}{r^{s_2+w-\alpha}}\big) \bigg] v^{\frac{s_1+s_2-1}{2}} d v dw ds_1ds_2+O((pq)^{-A}). \end{multline}
 We shift the $(s_1)=2$ to the line $1/2+\delta.$ There are no poles encountered in this contour shift. 
 
 The same archimedean analysis as is done in Section \ref{vil} simplifies this term to
 
\begin{multline}\label{eq:resr} \frac{\sqrt{pq}  \psi(p)\chi(-q)}{(pq)^2}
\sum_m  (\frac{1}{2\pi i})^3 \int_{(s_1)=1/2+\delta} \int_{(s_2)=1/2-\delta}  (\frac{l_1l_2}{k^2})^{\frac{s_1+s_2}{2}} \int_{(\alpha)=1/2}  \frac{\Gamma(\alpha)\Gamma(s_2-\alpha)}{\Gamma(s_2)} \frac{(pq)^{s_1} (pq)^{s_2}}{p^{s_1+s_2-\alpha}q^{s_2}} \\   \frac{L(\alpha,\chi)L(s_1+s_2-\alpha,\psi)}{L(s_1+2s_2-\alpha,\chi\psi)} \int_{(w)=\epsilon}   D(s_1,s_2,w)    \bigg[ \sum_{e=1} \frac{\mu(e)\chi(e)}{e^{\alpha+s_1+2s_2+1+w}} \sum_{d=1} \frac{1}{d^{s_1+s_2+1+w}} \sum_{a=1} \frac{\mu(a)}{a^{s_1+s_2+1+w}} \\ \sum_{b=1}  \frac{\mu^2(b)\chi(b)}{b^{s_1+2s_2+1+w}}  \sum_{\substack{n \in \Z\\ n \neq0}} \frac{\overline{\psi(n)} }{n^{\alpha+1+w}} \big(\sum_{\substack{r(baedn)\\ \{r: \frac{l_1l_2r^2}{\dd k^2}-mr+\dd\equiv 0(baedn)\}}} \frac{1}{r^{s_2+w-\alpha}}\big) \bigg] dw ds_1ds_2+O((pq)^{-A}), \end{multline} with again $D(s_1,s_2,w)$ defined as in \eqref{eq:ddd}.   

The $r$-sum has at most two terms in it following \eqref{eq:rru}, so it is bounded by $2(baedn)^{\delta}.$ The entire sum inside the brackets of \eqref{eq:resr} is absolutely convergent and is bounded (in terms of $p,q$) by $(pq)^{\epsilon}.$ Using \eqref{eq:con} gives \begin{equation}\label{eq:lfl}
\frac{L(\alpha,\chi)L(s_1+s_2-\alpha,\psi)}{L(s_1+2s_2-\alpha,\chi\psi)}\ll p^{\theta_1+\epsilon} q^{\theta_2+\epsilon}
\end{equation} with $\theta_1,\theta_2 \leq 1/4.$ The same analysis is done on $D(s_1,s_2,w)$ as in the previous section. The only difference between the analysis of this section and the previous section is the $\alpha$-integral. We can use an application of Stirling's approximation to the ratio of Gamma functions as in \cite{Y} to get  $$\frac{\Gamma(\alpha)\Gamma(s_2-\alpha)}{\Gamma(s_2)} \ll |1+t_2(1-\frac{t}{t_2})|^{-1/2}$$ where $\alpha=1/2+it$ and $s_2=1/2-\delta+it_2$ with $t,t_2$ large. A standard integration by parts argument in the $\alpha$-integral ensures that it converges. 

 So analogously to getting the bound \eqref{eq:resl}, \eqref{eq:resr} is bounded by \begin{equation}\label{eq:remrr}\frac{p^{\theta_1}q^{\theta_2+\delta}(pq)^{\epsilon}}{\sqrt{pq}} \sum_{l_1,l_2 \leq L} x_{l_1} x_{l_2}l_1^{1-\delta/2} l_2^{1-\delta/2}\sum_{k|(l_1,l_2)} \frac{1}{k^{3/2-\delta/2}} \sum_{\dd|{\frac{l_1l_2}{k^2}}}\frac{1}{{\dd}^{3/4}}\ll \frac{p^{\theta_1}q^{\theta_2+\delta}(pq)^{\epsilon}}{\sqrt{pq}} L^{1-2\delta+\epsilon}||x||_2^2 \end{equation} with $\theta_1,\theta_2 \leq 1/4.$ 
  
 
 Combining \eqref{eq:resl} and \eqref{eq:resr} we get the Kloosterman-Kloosterman term with $n\neq0$ is bounded by \begin{equation}\label{eq:finn}
\frac{p^{\theta_1+\epsilon}q^{\theta_2+\epsilon}}{\sqrt{pq}} L^{1+\epsilon}  ||x||_2^2\left( 1+ q^{\delta}\right).
\end{equation}

\section{Delta-Delta term}\label{dd}
The term from \eqref{eq:poisson} that we now study is \begin{multline}\label{eq:ddel} \sum_{l_1,l_2 \leq L} x_{l_1} x_{l_2}\sum_{q|(l_1,l_2)}  \sum_{\dd|{\frac{l_1l_2}{k^2}}} \frac{1}{{\dd}^{1/2}} \sum_{m \geq 1} \frac{1}{m^{1/2}}F'_{M'}(\frac{\dd m}{pq})  \sum_{n \geq 1} \frac{1}{n^{1/2}} F_M(\frac{n}{pq}) D_{n,\frac{l_1 l_2m}{\dd k^2}}(V)D_{n,\dd m}(W)\end{multline}
Now this term is non-zero only if $n=\frac{l_1 l_2m}{\dd k^2}$ and $n=\dd m,$ which implies $\sqrt{l_1l_2}=q\dd.$ So we can bound \eqref{eq:ddel} by  \begin{equation}\label{eq:ddel0} \sum_{l_1,l_2 \leq L} x_{l_1} x_{l_2}\sum_{k|(l_1,l_2)}   \sum_{\substack{\dd|{\frac{l_1l_2}{k^2}}\\ \sqrt{l_1l_2}=q\dd}} \frac{1}{{\dd}} \sum_{m \geq 1} \frac{1}{m}F'_{M'}(\frac{\dd m}{pq}) F_M(\frac{\dd m}{pq}). \end{equation}

Using a standard Mellin inversion we can reduce to $$\left[\int_0^\infty F'_{M'}(t)F_{M}(t)\frac{dt}{t} \right] \sum_{l_1,l_2 \leq L} \frac{x_{l_1} x_{l_2}}{(l_1l_2)^{1/2}}\sum_{k|(l_1,l_2)} k  + O((pq)^{-M}),$$ for any $M>0.$

An easy estimate for this last line is $(pq)^{\epsilon}L^{\epsilon}||x||_2^2.$

\section{Delta-Kloosterman term}\label{dk}
It is obvious the analysis for this term will suffice for the Kloosterman-delta term. 
The term is  \begin{multline}\label{eq:dkk} \sum_{l_1,l_2 \leq L} x_{l_1} x_{l_2}\sum_{k|(l_1,l_2)}  \sum_{\dd|{\frac{l_1l_2}{k^2}}} \frac{1}{{\dd}^{1/2}} \sum_{m \geq 1} \frac{1}{m^{1/2}}F'_{M'}(\frac{\dd m}{pq}) \sum_{n \geq 1} \frac{1}{n^{1/2}} F_M(\frac{n}{pq}) D_{n,\frac{l_1 l_2m}{\dd k^2}}(V)\times \\ \bigg(\sum_{c_2\equiv 0(q)} \frac{S_{\psi}(n,\dd m,c_2)}{c_2}W(\frac{4\pi \sqrt{n \dd m}}{c_2})\bigg)\end{multline}
Similar to the previous section, the only non-zero term is when $n=\frac{l_1 l_2m}{\dd k^2},$ which then reduces \eqref{eq:dkk} to  \begin{multline}\label{eq:dkk0} \sum_{l_1,l_2 \leq L} \frac{x_{l_1} x_{l_2}}{(l_1l_2)^{1/2}}\sum_{k|(l_1,l_2)} k \sum_{\dd|{\frac{l_1l_2}{k^2}}} \sum_{m \geq 1} \frac{1}{m}F'_{M'}(\frac{\dd m}{pq})  F_M(\frac{\frac{l_1 l_2m}{\dd k^2}}{pq}) \times \\ \bigg(\sum_{c_2\equiv 0(q)} \frac{S_{\psi}(\frac{l_1 l_2m}{\dd k^2},\dd m,c_2)}{c_2}W(\frac{4\pi \sqrt{ \frac{l_1l_2m^2}{k^2}}}{c_2})\bigg)\end{multline} This calculation is very similar to \cite{R}. 

We isolate the $m$-sum  to get $$\sum_m \frac{1}{m} F'_{M'}(\frac{\dd m}{pq})  F_M(\frac{\frac{l_1 l_2m}{\dd k^2}}{pq}) \sum_{y(c_2)^{*}}   \psi(y) e(\frac{\frac{l_1 l_2my}{\dd k^2}+\dd m \overline{y}}{c_2})W(\frac{4\pi \sqrt{ \frac{l_1l_2m^2}{k^2}}}{c_2}).$$ This is very similar to \eqref{eq:msum8} above. Performing Poisson summation on the $m$-sum modulo $c_2,$ and including the $c_2$-sum we get \begin{equation}\label{eq:dkp}
\sum_{c_2\equiv 0 (q)} \frac{1}{c_2} \sum_m \sum_{\substack{y(c_2)^{*}\\ \frac{l_1l_2}{\dd k^2}y^2-my+\dd\equiv 0(c_2)}} \psi(y) \int_{-\infty}^\infty F'_{M'}(\frac{\dd t}{pq})  F_M(\frac{\frac{l_1 l_2 t}{\dd k^2}}{pq})W(\frac{4\pi \frac{t}{k}\sqrt{ l_1l_2}}{c_2})e(\frac{-mt}{c_2})\frac{dt}{t}.
\end{equation}

Let us change variables $t \to \frac{kt}{\sqrt{l_1l_2}},$ giving \begin{equation}\label{eq:dkp00}
\sum_{c_2\equiv 0 (q)} \frac{1}{c_2} \sum_m \sum_{\substack{y(c_2)^{*}\\ \frac{l_1l_2}{\dd k^2}y^2-my+\dd\equiv 0(c_2)}} \psi(y) \int_{-\infty}^\infty F'_{M'}(\frac{\frac{\dd kt}{\sqrt{l_1l_2}} }{pq})  F_M(\frac{\frac{\sqrt{l_1 l_2} t}{\dd k}}{pq})W(\frac{4\pi t}{c_2})e(\frac{-mqt}{\sqrt{l_1l_2}c_2})\frac{dt}{t}.
\end{equation}

 We bound the $y$-sum by 2 as in Section \ref{myoy} and define $$T_m(x):=\int_{-\infty}^\infty F'_{M'}(\frac{\dd t}{pq})  F_M(\frac{\frac{l_1 l_2 t}{\dd k^2}}{pq})W(\frac{4\pi \frac{t}{k}\sqrt{ l_1l_2}}{x})e(\frac{-mt}{x})\frac{dt}{t}.$$ Then interchanging the $m$-sum and $c_2$-sum with another application of Mellin inversion, \eqref{eq:dkp} is bounded by \begin{multline}
2 \sum_m \sum_{c_2\equiv 0 (q)} \frac{1}{c_2}T_m(c_2)=2 \sum_m (\frac{1}{2\pi i}) \int_{(\sigma)} \widetilde{T_m}(s) \frac{\zeta(1+s)}{q^{1+s}}ds=\\ 2 \sum_m \left[ \frac{\widetilde{T_m}(0)}{q} +  (\frac{1}{2\pi i}) \int_{(-\delta)} \widetilde{T_m}(s) \frac{\zeta(1+s)}{q^{1+s}}ds\right],
\end{multline}
where $\delta>0.$
The main term is $$2\sum_m \frac{1}{q} \int_0^\infty \int_{-\infty}^\infty F'_{M'}(\frac{\dd t}{pq})  F_M(\frac{\frac{l_1 l_2 t}{\dd k^2}}{pq})W(\frac{4\pi \frac{t}{k}\sqrt{ l_1l_2}}{x})e(\frac{-mt}{x})\frac{dt}{t} \frac{dx}{x}$$

As $F_M,F'_{M'},W$ are compactly supported $\sqrt{l_1l_2} \sim q\dd,$ the $t$-integral and $x$-integral is of bounded size (independent of $p,q$ due to the multiplicative Haar measures). The $m$-sum converges by an integration by parts argument. Also the $m$-sum can be truncated at length $O( pq),$ with a remainder term having a bound $O((pq)^{-M}),$ for any $M \in \N.$ Integrate by parts once in the $t$-integral to gain a term of size $\dd \sim \frac{\sqrt{l_1l_2}}{q}.$ The size of the integral is of size $\frac{pq}{\dd}$, and so bounding the $m$ trivially gets the whole main term is of size $O(p^{\epsilon}q^{\epsilon-1} ).$ The remainder term by an analogous argument gives $O(p^{\epsilon}q^{\epsilon-1} )$ also.



 Our leftover analysis is then \begin{equation}\label{eq:dka}
p^{\epsilon}q^{\epsilon-1} \sum_{l_1,l_2 \leq L} (l_1l_2)^{-1/2} x_{l_1} x_{l_2}\sum_{q|(l_1,l_2)} q\ll p^{\epsilon}q^{\epsilon-1}L^{\epsilon}||x||_2^2.
\end{equation}

Now one can apply the same argument as we did for the delta-Kloosterman term to the Kloosterman-delta term, and get analogous bound $O(q^{\epsilon}p^{\epsilon-1} ),$ and bounding the $l_1,l_2$ sum gives $q^{\epsilon}p^{\epsilon-1}L^{\epsilon}||x||_2^2.$


\section{Positivity of test functions on spectral side of trace formula}\label{isomo}

In order to go from estimating \begin{equation}\label{eq:2app00}
\sum_{f\in H(p,\chi)}\frac{1}{(f,f)} h(V,t_f)|\sum_{l \leq L} x_l \lambda_f(l)|^2 \sum_{h \in H(q,\psi) }\frac{1}{(h,h)}h(W,t_h) |L(1/2,f \times h)|^2\end{equation} to a single term \begin{equation}\label{eq:isoyo}\frac{1}{(f,f)(h,h)} h(V,t_f)h(W,t_h) |\sum_{l \leq L} x_l \lambda_f(l)|^2 |L(1/2,f \times h)|^2,\end{equation} we need careful choices of our test functions $V,W.$ We use the following proposition of \cite{Venk1}:

\begin{prop}\label{venkm2}
Given $j_0 \in \mathbb{N}, \epsilon > 0$ and an integer $N >0,$ there is a $V$ of compact support so that $h(V,t_j)=1,$ and for all $j' \neq j_0, h(V,t_{j'}) \ll \epsilon(1+|t_{j'}|)^{-N},$ and for all $k$ odd, $h(V,k) \ll \epsilon k^{-N}.$ 

Given $k_0, \epsilon >0$ and an integer $N>0,$ there is a $V$ of compact support so that $h(V,k_0)=1, h(V,k) \ll \epsilon k ^{-N}$ for $k$ odd $k\neq k_0,$ and $h(V,t) \ll (1+|t|)^{-N}$ for all $\mathbb{R}.$ 
\end{prop}

Using the above proposition we can choose $V,W$ to isolate spectral parameters $\{t_{f^{\circ}},t_{h^{\circ}}\}$ with $N_V,N_W$ large enough and $\epsilon=\epsilon_V=\epsilon_W$ small enough so that the sum  \begin{equation}\label{eq:2app007}
\sum_{\substack{f\in H(p,\chi)\\ t_f \neq t_{f^{\circ}}}}\frac{1}{(f,f)} h(V,t_f)|\sum_{l \leq L} x_l \lambda_f(l)|^2 \sum_{\substack{h \in H(q,\psi)\\ t_h \neq t_{h^{\circ}}}}\frac{1}{(h,h)}h(W,t_h) |L(1/2,f \times h)|^2 < \epsilon. \end{equation}

With this choice of $V,W$ \eqref{eq:isoyo} we have the inequality  \begin{multline}\label{eq:isoyo0}\frac{1}{(f^{\circ},f^{\circ})(h^{\circ},h^{\circ})}  |\sum_{l \leq L} x_l \lambda_f(l)|^2 |L(1/2,f^{\circ} \times h^{\circ})|^2\ll \\ \sum_{f\in H(p,\chi)}\frac{1}{(f,f)} h(V,t_f)|\sum_{l \leq L} x_l \lambda_f(l)|^2 \sum_{h \in H(q,\psi) }\frac{1}{(h,h)}h(W,t_h) |L(1/2,f \times h)|^2. \end{multline}

\section{Putting it all together}

Using the estimates from Sections \ref{zero}, \ref{remb}, \ref{dd}, and \ref{dk}, and letting  $\theta=\theta_1=\theta_2$ in \eqref{eq:finn} as well as \eqref{eq:zerothe}, we get (putting explicitly in the orthonormal basis built into the trace formula)
\begin{multline}\label{eq:2app}
\sum_{f\in H(p,\chi)}\frac{1}{(f,f)} h(V,t_f)|\sum_{l \leq L} x_l \lambda_f(l)|^2 \sum_{h \in H(q,\psi) }\frac{1}{(h,h)}h(W,t_h) |L(1/2,f \times h)|^2 \ll \\ ||x||_2^2\bigg( (pq)^{\epsilon}L^{\epsilon}\{\text{delta-delta term}\} \\
+p^{\epsilon}q^{\epsilon-1}L^{\epsilon}\{\text{delta-Kloosterman term}\} +q^{\epsilon}p^{\epsilon-1}L^{\epsilon}\{\text{Kloosterman-delta term}\}\\
+(pq)^{\theta-1+\epsilon} L^{\epsilon} +(pq)^{\theta -1/2 + \epsilon} (1+q^{\delta})L^{1+\epsilon}\{\text{Kloosterman-Kloosterman term}\}. \bigg)
\end{multline}

\begin{remark}

\begin{enumerate}
\item If we take $L=1,$ the trivial amplifier, we obtain the expected Lindel\"{o}f on average bound even with $\theta=1/4.$

\item Remember the ``abnormal" term $(1+q^\delta)$ in the Kloosterman Kloosterman estimate comes from a contour shift to get the ``main" term and ``remainder" term in \eqref{eq:ymr} and \eqref{eq:yrem}. The Kloosterman-Kloosterman bounds above in \eqref{eq:2app} come from these respective terms' analysis in \eqref{eq:resl} and \eqref{eq:remrr}. We expect with more refined analysis and exploiting some kind of symmetry or functional equation, similar to what is done in \cite{So}, the ``remainder" term is much smaller than the ``main" term.
\item The term $(1+q^\delta)$ being not symmetric in $p$ and $q$ is an artifact of that we are only amplifying the forms of level $p$ in \eqref{eq:2app}. If we put an amplifier on both forms $f\in H(p,\chi)$ and $h \in H(q,\psi)$ we would again expect a symmetry in the estimate.
\item Note in this paper, we do not apply any Ramanujan bounds towards the Fourier coefficients as all the analysis is done on the geometric side of the trace formula. The only crucial ingredient is the subconvexity bounds for $GL_1$ L-functions. 

\end{enumerate}
\end{remark}

As mentioned, it is implied in the trace formulas used above that the $f$- and $h$-sums are orthonormal with respect to the Petersson inner product. This inner product also depends on the level $p$ and $q.$ Using that $(f,f)\ll p (\log p)^3$ and $(h,h)\ll q (\log q)^3,$ and a choice of $V,W$ as in section \ref{isomo} to get the inequality \eqref{eq:isoyo0} we get  \begin{multline*}|\sum_{l \leq L} x_l \lambda_{f^{\circ}}(l)|^2 |L(1/2,f^{\circ} \times h^{\circ})|^2 \ll  ||x||_2^2\bigg((pq)^{\theta+\epsilon} L^{\epsilon}  +(pq)^{\theta+1/2 + \epsilon} (1+q^{\delta})L^{1+\epsilon}+ \\ (pq)^{1+\epsilon}L^{\epsilon}+p^{1+\epsilon}q^{\epsilon}L^{\epsilon}+q^{1+\epsilon}p^{\epsilon}L^{\epsilon} \bigg).\end{multline*}

We now choose the coefficients as in \cite{KMV}, 

\begin{equation}
x_l
 := \left\{ \begin{array}{ll}
        \lambda_{f^{\circ}}(l) & \text{if } l \text{ is prime} \leq L^{1/2} \\
        -1 & \text{if } l \text{ is a square of a prime} \leq L^{1/2} \\
        0 & \text{ otherwise}.\end{array} \right.  \end{equation}

Now using the fact that $$|\sum_{l \leq L} x_l \lambda_{f^{\circ}}(l)|^2 \gg(pq)^{-\epsilon}L$$ and $$||x||_2^2 \ll L^{1/2+\epsilon},$$ we have \begin{equation}\label{eq:useamp}
|L(f^{\circ} \times h^{\circ}, 1/2)| \ll (pq)^{\epsilon}\bigg((pq)^{\frac{\theta}{2}}L^{-1/4}+(1+q^{\delta})^{1/2}(pq)^{\frac{\theta}{2}+\frac{1}{4}}L^{\frac{1}{4}}+\frac{(pq)^{1/2}}{L^{1/4}}+\frac{p^{1/2}+q^{1/2}}{L^{1/4}}\bigg).
\end{equation}

We can simplify it to \begin{equation}\label{eq:useamp0}
|L(f^{\circ} \times h^{\circ}, 1/2)| \ll (pq)^{\epsilon}\bigg(\frac{(pq)^{\frac{\theta}{2}}+(pq)^{1/2}+p^{1/2}+q^{1/2}}{L^{1/4}}+(1+q^{\delta})^{1/2}(pq)^{\frac{\theta}{2}+\frac{1}{4}}L^{\frac{1}{4}}\bigg).
\end{equation}

If we choose  $L=(pq)^{\frac{1}{2}-\theta}$ then \eqref{eq:useamp} is bounded by $$(pq)^{\epsilon}\bigg((pq)^{\frac{3\theta}{2}-1/8}+(pq)^{3/8+\frac{\theta}{4}}+p^{3/8+\frac{\theta}{4}}q^{\frac{\theta}{4}-1/8}+q^{3/8+\frac{\theta}{4}}p^{\frac{\theta}{4}-1/8} +  (1+q^{\delta})(pq)^{\frac{3}{8}+\frac{\theta}{4}}\bigg).$$ 


This completes Theorem \ref{theo}.

\end{document}